\documentclass[a4,11pt]{article}
\usepackage[utf8]{inputenc}
\usepackage{epsfig}
\usepackage{amsmath}
\usepackage{amssymb}
\usepackage{amsthm}
\usepackage{mathrsfs}

\renewcommand{\Re}{\mathop{\rm Re}}
\renewcommand{\Im}{\mathop{\rm Im}}
\newcommand{\R}{\mathfrak{t}}
\newcommand{\s}{\mathfrak{s}}

\theoremstyle{plain}
\newtheorem{theorem}{Theorem}[section]

\newtheorem{proposition}[theorem]{Proposition}
\newtheorem{corollary}[theorem]{Corollary}

\newtheorem{remark}[theorem]{Remark}

\begin{document}

\title{\Large{On semi-classical weight functions on the unit circle}}

\author
{Cleonice F. Bracciali$^{a,}$\thanks{cleonice.bracciali@unesp.br (corresponding author)} , \
Karina S. Rampazzi$^{a,}$\thanks{karina.rampazzi@unesp.br.} ,  \
Luana L. Silva Ribeiro$^{a, b,}$\thanks{luanadelimasr@gmail.com.} \\[1ex]
  {\small $^{a}$Departamento de Matem\'{a}tica, IBILCE, UNESP\,-\,Universidade Estadual Paulista,} \\
  {\small 15054-000, S\~{a}o Jos\'{e} do Rio Preto, SP, Brazil.}\\[1ex]
   {\small $^{b}$Departamento de Matem\'{a}tica, UnB\,-\,Universidade de Bras\'{\i}lia,} \\
  {\small 70910-900, Bras\'{\i}lia, DF, Brazil.}\\[1ex] 
}

\maketitle

\thispagestyle{empty}

\begin{abstract}
We consider orthogonal polynomials on the unit circle  associated with certain semi-classical weight functions.
This means that the Pearson-type differential equations satisfied by these weight functions involve two polynomials of degree at most 2. 
We determine all such semi-classical weight functions and this also includes an extension of the Jacobi weight function on the unit circle. 
General structure relations for the orthogonal polynomials and non-linear difference equations for the associated complex Verblunsky coefficients are established. 
As application, we present several new structure relations and non-linear difference equations associated with some of these semi-classical weight functions.
\end{abstract}

{\noindent}Keywords:  Orthogonal polynomials on the unit circle; semi-classical weight functions; difference equations; structure relations. \\

{\noindent}2020 Mathematics Subject Classification: 42C05, 33C47

\setcounter{equation}{0}
\section{Introduction} 
Let $\mu$ be a positive measure on the unit circle $\mathbb{T} = \{z \in \mathbb{C} : |z|=1 \}$. The
associated sequence of orthonormal polynomials on the unit circle, $\{\Psi_n\}_{n\geqslant 0}$, satisfies
\begin{equation*}
\langle \Psi_n,\Psi_m  \rangle = \int_{\mathbb{T}} \Psi_n (z)\overline{\Psi_m(z)}  d \mu (z)  =
\int_{0}^{2\pi} \Psi_n (e^{i\theta})\overline{\Psi_m(e^{i\theta})}  d \mu (e^{i\theta})  = \delta_{n,m},  
\end{equation*}
where $\Psi_n (z) =  \kappa_n z^n + \cdots $, with $\kappa_n>0.$
The monic orthogonal polynomials on the unit circle (MOPUC, in short) are denoted by
$ \Phi_n(z)$, i.e., $ \Phi_n(z) = \Psi_n(z) / \kappa_n$. Hence
$||\Phi_n ||^2 = \int_{\mathbb{T}} |\Phi_n (z)|^2 d \mu (z)= \kappa_n^{-2}$. For recent basic literature  on this subject see, for example, \cite{Is05} and \cite{Simon-Book-p1}.

These polynomials satisfy the relation
\begin{equation}  \label{Szegorecorrence}
\Phi_{n}(z) = z\Phi_{n-1}(z) - \overline{\alpha}_{n-1} \Phi_{n-1}^{*}(z), \quad n \geqslant 1,
\end{equation}
where $\Phi_{0}(z) =1,$ $\Phi_{n}^*(z) =z^n \overline{\Phi_{n}(1/\overline{z})}$ is the reciprocal polynomial and $\alpha_{n-1} = - \overline{\Phi_{n}(0)}$ are known as Verblunsky coefficients.
From relation \eqref{Szegorecorrence}, it is known that
$ \kappa_{n}^{2}/\kappa_{n+1}^{2} =  1-|\alpha_n|^2.$

We denote  the MOPUC as
\begin{equation}  \label{pol_expli}
\Phi_{n}(z) =  z^n + \gamma_{n} z^{n-1}  + \cdots + \beta_{n} z - \overline{\alpha}_{n-1}, \quad  n \geqslant 0,
\end{equation}
with $\alpha_{-1} = -1$.
Using relation \eqref{Szegorecorrence} it is easy to show that the coefficients $\gamma_{n}$ in \eqref{pol_expli} are given in terms of the Verblunsky coefficients as 
\begin{equation}  \label{lnn1}
\gamma_{n} = \sum_{j=0}^{n-1} \overline{\alpha}_{j} \alpha_{j-1}, \ n \geqslant 1, \quad \mbox{and} \quad \gamma_0=0.
\end{equation}
Notice that $\gamma_n = \gamma_{n-1} + \overline{\alpha}_{n-1} \alpha_{n-2}.$  
Taking the derivative of \eqref{Szegorecorrence} and evaluating in $z=0$, one can see that the coefficients $\beta_{n} = \Phi^{\prime}_{n}(0)$ satisfy
\begin{equation}  \label{rnn1}
\beta_{n}  = -(\overline{\alpha}_{n-2} + \overline{\alpha}_{n-1} \overline{\gamma}_{n-1}) =  -[(1-|\alpha_{n-1}|^2)\overline{\alpha}_{n-2}+\overline{\alpha}_{n-1}\overline{\gamma}_{n}], \quad n \geqslant 1.
\end{equation}
The latter right hand side equality comes from \eqref{lnn1}.

Since $\Phi_{n}(z) =  z^n + \gamma_{n} z^{n-1}  + O(z^{n-2}),$
hence $\langle \Phi_{n-1},\Phi_{n} \rangle =0$ yields
\begin{equation} \label{fato1}
\langle \Phi_{n-1},z^{n} \rangle =- \overline{\gamma}_{n} \langle \Phi_{n-1},\Phi_{n-1} \rangle.
\end{equation}

In this work, we consider $\mu$ to be absolutely continuous and we set $d \mu(e^{i\theta}) = \nu(e^{i\theta}) d\theta. $
Therefore, we can write
\begin{equation*}
\langle f(z), z^n \rangle 
= \int_{0}^{2\pi} f(e^{i\theta}) e^{-i n \theta} \nu(e^{i\theta}) d\theta  = \int_{\mathbb{T}} f(z) z^{-n} \nu(z) \frac{dz}{iz}.
\end{equation*}

For simplicity, we denote $w(\theta)=\nu(e^{i\theta})$. 
The moments of the weight function $w$ are given by
$  \mu_n =  \langle 1,z^{n} \rangle =
\int_{0}^{2\pi} e^{-i n \theta} w (\theta) d\theta, $
for $n \in \mathbb{Z}$.

\medskip

In this paper we analyse the properties of MOPUC associated with semi-classical weight function on the unit circle. 
A weight function on the unit circle $w$ is called semi-classical if 
\begin{equation}\label{Eq-Tipo-Pearson-1}
\frac{d}{d\theta} \left[ A (e^{i\theta})  w(\theta) \right]  = B(e^{i\theta})  w(\theta).
\end{equation}
where the functions $A$ and $B$ are Laurent polynomials and 
$A(e^{i\theta})=0$ at the singular points of $1/w$, 
see \cite{AlMa91} and  \cite{Ma00}.
The  equation \eqref{Eq-Tipo-Pearson-1} will be called a Pearson-type differential equation.

The relation \eqref{Eq-Tipo-Pearson-1} can also be written as
\begin{equation}  \label{Eq-Tipo-Pearson-2}
\frac{d w(\theta)/ d \theta}{w(\theta)} = \frac{w'(\theta)}{w(\theta)}=\frac{B(e^{i\theta})-ie^{i\theta}\frac{dA(e^{i\theta})}{d \, e^{i\theta}}}{A(e^{i\theta})}.
\end{equation}  

For orthogonality on the unit circle, the problem of classifying semi-classical OPUC  usually is studied considering linear functionals that satisfy a Pearson-type equation, see for instance,  \cite{CaSu97}, \cite{Su01},  and  \cite{Su08}. This approach can be adapted to the equation \eqref{Eq-Tipo-Pearson-1} by considering that if $\deg A(z) = p$ and $\max\{p-1, \deg((p-1)A(z) + iB(z))\}=q,$ then $w$ belongs to the class $(p,q)$ of semi-classical weight functions. A weight function that belongs to the class $(p,q)$ also  belongs to the class $(p+1,q+1).$

In the theory of orthogonal polynomials on the real line, it is well known that if the weight function $\omega$, defined on an  interval $(a,b)$, satisfies the Pearson equation $[\sigma(x)  \omega(x)]^{\prime} = \rho(x)  \omega(x)$, where $\sigma$ is polynomial of degree at most 2 and $\rho$ is a polynomial of degree 1, then the weight function  is classical, and the associated orthogonal polynomials are the classical orthogonal polynomials. If $\sigma$ is polynomial of degree greater than 2 or $\rho$ is a polynomial of degree different from 1, the weight function  is known as semi-classical.
If $\sigma\omega$ vanishes at the extreme points of the orthogonality interval, the orthogonal polynomials associated with classical or semi-classical weight function on the real line satisfy differential-recurrence relations (also known as structure relations). See, for example, \cite{BLN87, Bra96, Maroni84, Maroni90, Va08}. 

The recurrence coefficients of orthogonal polynomials on the real line associated with certain semi-classical weight functions (for example, the Freud weight function) satisfy difference equations and are related to some Painlev\'{e} equations. See \cite{Fr76, Ma86, Va08}.

A known example of semi-classical weight function on the unit circle is $w(\theta) = e^{t \cos(\theta)},$ with $t>0$,  see \cite{Is05}. The associated MOPUC satisfy the structure relation
\begin{equation}  \label{rel_est_conhecida}
\Phi_{n}^{\prime}(z) = n \Phi_{n-1}(z)
+ \frac{t \kappa_{n-2}^2}{2 \ \kappa_{n}^2} \Phi_{n-2}(z), \quad n \geqslant 2.
\end{equation}
The Verblunsky coefficients are real and they depend on the parameter $t$. From the structure relation for the MOPUC, it is possible to show that the Verblunsky coefficients  satisfy the non-linear difference equation
\begin{equation}  \label{Painleve}
\alpha_{n}(t) + \alpha_{n-2}(t) = -\frac{2n}{t} \frac{\alpha_{n-1}(t)}{1-\alpha_{n-1}^2(t)}, \quad n \geqslant 2.
\end{equation}
This well known equation appeared in Periwal and Shevitz  \cite{PS90}.
It corresponds to the Painlev\'{e} discrete equation ${\rm dP_{II}}$, see \cite{Va08} and the references therein.
Studies about semi-classical weight functions on the unit circle can be found also in \cite{BraRe12, CaSu97, Ma00, Su08}.

The objective of this paper is to investigate semi-classical weight functions on the unit circle, $w$, whose MOPUC have complex Verblunsky coefficients. 
We restrict this investigation to semi-classical weight functions for which  $A$ and $B$ in \eqref{Eq-Tipo-Pearson-1} are polynomials of degree at most 2.
We remark that in \cite{CaSu97} the authors investigated the particular case for which  $A$ is a polynomial of degree exactly 2 and it has distinct zeros. Furthermore, in \cite{Su01} the case $A(z) = (z-r)^2$, with $B(r) - i r A^{\prime}(r)=0,$ has been considered.

One goal  here is to determine all semi-classical weight functions on the unit circle that belong to this class.  Other aim is to present structure relations for the associated MOPUC and  non-linear difference equations for the associated complex Verblunsky coefficients, for this class of semi-classical weight functions on the unit circle. We would like to point out that in \cite{IsWi01} a general analysis of ladder operators associated with a weight function of the type $w (\theta) = e^{-v(\theta)}$ was established and structure relations for the associated OPUC were  presented. 

This paper is structured as follows. In Section \ref{sec_WF} we use the definition of semi-classical weight function on the unit circle \eqref{Eq-Tipo-Pearson-1} to set the form of the structure relation for the associated MOPUC, see Theorem \ref{prop_rel_est}. In Theorem \ref{TEO_eq_de_dif_caso_geral} we present two non-linear difference equations  for the associated complex Verblunsky coefficients, which is one main result.

In Section \ref{sec_Characterization}
we determine all semi-classical weight functions on the unit circle, $w$, for which $A$ and $B$ in \eqref{Eq-Tipo-Pearson-1} are polynomials of degree at most 2.   
 The method used is to consider all possible zeros of polynomial $A$ such that $w$ is positive.
Many known semi-classical weight functions were determined and we also found the semi-classical weight function 
\begin{equation*}
w(\theta)= \tau(\lambda,\beta,\eta) \, e^{-\eta \theta}[\sin^2(\theta/2)]^{\lambda}[\cos^2(\theta/2)]^{\beta},
\end{equation*}
where $\eta \in \mathbb{R}$, $\lambda >-1/2$, $\beta >-1/2$ and  $\tau(\lambda,\beta,\eta)$ is a constant, see \eqref{newfunction}. 
Particular cases of this weight function cover most of the semi-classical weight functions known in the literature, see Remark \ref{remark36}. 
However,   to the best of our knowledge, 
the case $\lambda\beta\eta \neq 0$ has never been considered before.

As application of these results, in Section \ref{sec_Application}, we present structure relations for the MOPUC and non-linear difference equations for the Verblunsky coefficients, using the semi-classical weight functions described in Section \ref{sec_Characterization}. 
In particular, in Subsection \ref{subsection_exemplo1} we present several structure relations and several non-linear difference equations for the complex Verblunsky coefficients associated with
$w(\theta)= e^{-\theta\eta}[\sin^2(\theta/2)]^{\lambda},$
$\eta \in \mathbb{R}$, $\lambda >-1/2$. 
Subsection \ref{subsection42} is dedicated to the weight function $w(\theta)= e^{2|u| \sin(\theta + \arg  u)}$, $u \in \mathbb{C}$, where we present the non-linear difference equation \eqref{ED_iu}, that can be seen as a complex extension of the Painlev\'{e} discrete equation 
${\rm dP_{II}}$.
 
\setcounter{equation}{0}
\section{Structure relations and non-linear difference equations} 
\label{sec_WF}

Since we assume $A$ and $B$ to be complex polynomials of degree at most $2$, we set
\begin{equation*}
 A(z) = a_{2} z^2 + a_{1}z + a_{0} \quad \mbox{and} \quad
 B(z) = b_{2} z^2 + b_{1}z + b_{0}.
\end{equation*}
Thus the Pearson-type equation \eqref{Eq-Tipo-Pearson-2} becomes
\begin{equation} \label{Eq-Tipo-Pearson-3}
\frac{dw(\theta)/d\theta}{w(\theta)} 
 = \frac{(b_{2}-2ia_{2})z^2+(b_{1}-ia_{1})z+b_{0}}{a_{2}z^2+a_{1}z+a_{0}}, \ \ z=e^{i\theta}.
\end{equation}
Equivalently, we can write
\begin{equation} \label{Eq-Tipo-Pearson-4}
i\frac{dw(\theta)/dz}{w(\theta)}= \frac{(b_{2}-2ia_{2})z^2+(b_{1}-ia_{1})z+b_{0}}{z(a_{2}z^2+a_{1}z+a_{0})}, \ \ z=e^{i\theta}.
\end{equation}

Before introducing our results we need some results from  Magnus \cite{Ma00}. For example, if $w$ satisfies \eqref{Eq-Tipo-Pearson-2} with $A$ and $B$  polynomials of degree at most $d$, with $d \geqslant 1$, and $w(0)=w(2\pi)$, then
$
\langle A\Phi_{n}',z^{k} \rangle = \langle \Phi_{n} [iB(z)+(k+1)A(z)], z^{k+1} \rangle,$ $n \geqslant 2$ 
and 
$\langle A\Phi_{n}',z^{k} \rangle =0,$ for $  k=d-1,d,\ldots,n-2.$

Indeed, using integration by parts and property \eqref{Eq-Tipo-Pearson-2}, one can see that
\begin{equation} \label{Aphilinhazk}
\langle A\Phi_{n}',z^{k} \rangle  
 = \langle \Phi_{n}(z) [iB(z)+(k+1)A(z)], z^{k+1} \rangle
- i \Phi_{n}(1) A(1)[w(2\pi) - w(0)],
\end{equation}
for $k=0,1,\ldots,n$. However, for $k=d-1,d,\ldots,n-2$, the polynomial $[iB(z)+(k+1)A(z)]z^{-k-1}$ $ \in \mbox{Span}\{z^{-(n-1)}, z^{-(n-2)}, \ldots, $ $ z^{-1}\}$,  hence $\langle \Phi_{n} [iB(z)+(k+1)A(z)], z^{k+1} \rangle=0.$

In this work, we consider $d=2$ and the condition $A(1)[w(2\pi) - w(0)]=0$, hence  
\begin{equation} \label{orthogonality-A}
\langle A \Phi_{n}' ,z^{k} \rangle =0, \quad k=1,2,\ldots,n-2.
\end{equation}

Magnus in \cite{Ma00} has shown that if $A$ and $B$ are polynomials of degree  $\leqslant d$, then
$ A\Phi_{n}'$  is a linear combination of 
$ \Phi_{n+d-1}, \Phi_{n+d-2},\ldots, \Phi_{n-1}$ and  
$ \Phi_{n+d-1}^{*}, z\Phi_{n+d-2}^{*},\ldots, z^{d-2}\Phi_{n+1}^{*}$.

\medskip

The next result presents structure relations of the  MOPUC in terms of the polynomials $\Phi_{n+1}$, $\Phi_{n}$, $\Phi_{n-1},$ and $\Phi_{n}^{*}$ (instead of  $\Phi_{n+1}^{*}$)  when the associated weight function satisfies \eqref{Eq-Tipo-Pearson-2} with  $A$ and $B$ polynomials of degree at most 2 and $A(1)[w(2\pi) - w(0)]=0$. All the coefficients are given explicitly.

\begin{theorem} \label{prop_rel_est}
Consider a weight function that satisfies \eqref{Eq-Tipo-Pearson-2}, where  $A$ and $B$ are polynomials of degree at most 2 and $A(1)[w(2\pi) - w(0)]=0$. If 
$ A(z)=a_{2}z^2+a_{1}z+a_{0} $ and $ B(z)=b_{2}z^2+b_{1}z+b_{0},$
then the associated MOPUC  satisfy the structure relation given by
\begin{equation}  \label{Est-geral2-nova}
A(z)\Phi_{n}'(z)=
n a_2 \Phi_{n+1}(z)
+ \s_{n,n}\Phi_{n}(z)
+ \s_{n,n-1}\Phi_{n-1}(z)
+ \R_n \Phi_{n}^{\ast}(z), \  n \geqslant 2,
\end{equation}
where 
\begin{align}
\R_{n} &= (ib_{2}+a_{2})\overline{\alpha}_{n}, 
\label{An-nova} \\[0.5ex]
\s_{n,n-1} &= (ib_{0}+na_{0})(1-|\alpha_{n-1}|^2), 
\label{Bn_n-1} \\[0.5ex]
\s_{n,n} &=   n a_1 - a_2 \gamma_{n} +  [ib_{2}-(n-1)a_{2}]\overline{\alpha}_{n} \alpha_{n-1}, \label{antiga50}
\end{align}
$\alpha_n$ are the Verblunsky coefficients and $\gamma_{n}$ is given in \eqref{lnn1}. The coefficient $\s_{n,n}$ can also be written as
\begin{equation} \label{antiga49}
 \s_{n,n}  =  ib_1 + (n+1)a_1 - a_0 \overline{\gamma}_{n}
- [ib_0 + (n+1)a_0] \alpha_{n} \overline{\alpha}_{n-1}.  
\end{equation}
\end{theorem}
\begin{proof}
Observe that $A(z)\Phi_{n}'(z)- n a_2\Phi_{n+1}(z)-\R_{n}\Phi_{n}^{\ast}(z)$ is polynomial of degree at most $n$. Hence, one can write 
\begin{equation*}
A(z)\Phi_{n}'(z)-na_2\Phi_{n+1}(z)-\R_{n}\Phi_{n}^{\ast}(z) =\sum_{j=0}^n \s_{n,j} \Phi_{j}(z).
\end{equation*}

First, we consider
$\langle A\Phi_{n}', 1 \rangle -n a_2 \langle \Phi_{n+1}, 1 \rangle - \R_{n}\langle \Phi_{n}^{\ast}, 1 \rangle =
\sum_{j=0}^n \s_{n,j} \langle \Phi_{j}, 1 \rangle,$
and using the orthogonality properties, we obtain
\begin{equation*} 
 \langle A \Phi_{n}', 1 \rangle - \R_{n}\langle \Phi_{n}, \Phi_{n} \rangle = \s_{n,0} \langle \Phi_{0}, 1 \rangle.
\end{equation*} 
Setting
$   \R_{n} = \langle A\Phi_{n}', 1 \rangle/\langle \Phi_{n}, \Phi_{n} \rangle,$ we get  $\s_{n,0}=0$. 
From \eqref{Aphilinhazk} with $k=0$,
 \begin{align*}
\langle A\Phi_{n}', 1 \rangle
& = \langle \Phi_{n} [(ib_{2}+a_{2}){z^2}+(ib_{1}+a_{1}){z}+(ib_{0}+a_{0})], z \rangle \\
& = (ib_{2}+a_{2}) \langle z \Phi_{n}, 1  \rangle.
\end{align*}
From \eqref{Szegorecorrence} it follows that $\langle z \Phi_{n}, 1  \rangle =
\overline{\alpha}_{n} \langle\Phi_{n},\Phi_{n}\rangle$, 
and thus \eqref{An-nova} holds.

By successively doing $k=1,2,\ldots, n-2$, we have
\begin{equation*}
\langle A\Phi_{n}', z^k \rangle -n a_2 \langle \Phi_{n+1}, z^k \rangle - \R_{n}\langle \Phi_{n}^{\ast}, z^k \rangle =
\sum_{j=1}^n \s_{n,j} \langle \Phi_{j}, z^k \rangle
\end{equation*}
and, from the orthogonality properties and \eqref{orthogonality-A}, $\s_{n,k}=0,$ for $k=1,2,\ldots, n-2$. Hence,   
\begin{equation*}
A(z)\Phi_{n}'(z)-n a_2\Phi_{n+1}(z)-\R_{n}\Phi_{n}^{\ast}(z) = \s_{n,n-1} \Phi_{n-1}(z) + \s_{n,n} \Phi_{n}(z),
\end{equation*}
and \eqref{Est-geral2-nova} holds.

To find the explicit value of $\s_{n,n-1} $, we take the inner product of both sides of equation \eqref{Est-geral2-nova} with  $z^{n-1}$, and we obtain
$\langle A\Phi'_{n},z^{n-1} \rangle 
=  \s_{n,n-1} \langle \Phi_{n-1},z^{n-1} \rangle.$
From \eqref{Aphilinhazk} with $k=n-1$, we know that
\begin{align*}
\langle A\Phi'_{n},z^{n-1} \rangle  
& = \langle \Phi_{n}[(ib_{2}+na_{2}){z^2}+(ib_{1}+na_{1}){z}+(ib_{0}+na_{0})],z^{n} \rangle, \\[0.5ex]
& = (ib_{0}+na_{0}) \langle \Phi_{n},\Phi_{n}  \rangle.
\end{align*}
Hence, we proved \eqref{Bn_n-1}, i.e.,
$\s_{n,n-1}  = (ib_{0}+na_{0})\kappa_{n-1}^{2}/\kappa_{n}^{2} =(ib_{0}+na_{0})[1-|\alpha_{n-1}|^2].$

For the explicit value of $\s_{n,n}$, first we take the inner product of both sides of equation  \eqref{Est-geral2-nova} with $\Phi_{n}$, to find
\begin{equation} \label{aux_antiga50}
\langle A\Phi_{n}', \Phi_{n} \rangle = \s_{n,n} \langle \Phi_{n}, \Phi_{n} \rangle + \R_{n} \langle \Phi_{n}^{\ast}, \Phi_{n} \rangle.
\end{equation}
On the other hand, from the orthogonality properties we know that
\begin{equation*}
\langle A\Phi'_{n},\Phi_{n} \rangle 
 = n a_{2} \langle z^{n+1},\Phi_{n} \rangle + [\gamma_{n} (n-1) a_{2} + n a_{1}  ] \langle \Phi_{n},\Phi_{n}\rangle,
\end{equation*}
and from \eqref{fato1} we get
$\langle A\Phi'_{n},\Phi_{n} \rangle 
 = - \gamma_{n+1} n a_{2} \langle \Phi_n,\Phi_{n} \rangle + [\gamma_{n} (n-1) a_{2} + n a_{1}  ] \langle \Phi_{n},\Phi_{n}\rangle.$
Substituting these relations in \eqref{aux_antiga50}, we obtain
\begin{equation*}
\s_{n,n} = - \gamma_{n+1} n a_{2}  + \gamma_{n} (n-1) a_{2} + n a_{1}  + \R_n \alpha_{n-1},
\end{equation*}
using \eqref{lnn1} and \eqref{An-nova}, the equation \eqref{antiga50} holds.

Finally,  we use \eqref{Est-geral2-nova} to get
$\langle A\Phi'_{n},z^{n} \rangle  = \s_{n,n-1} \langle \Phi_{n-1},z^{n} \rangle + \s_{n,n} \langle \Phi_{n},z^{n}\rangle.$ Moreover, 
since $ \langle \Phi_{n}, z^{n} \rangle = \langle \Phi_{n},\Phi_{n} \rangle$ and \eqref{fato1}, it follows that
\begin{equation} \label{aux_antiga49}
\langle A\Phi'_{n},z^{n} \rangle  = - \s_{n,n-1} \overline{\gamma}_{n} \langle \Phi_{n-1},\Phi_{n-1} \rangle + \s_{n,n} \langle \Phi_{n},\Phi_{n} \rangle.
\end{equation}
From \eqref{Aphilinhazk} with $k=n$, we obtain
 \begin{align*}
\langle A\Phi_{n}', z^n \rangle
& = \langle \Phi_{n} [(ib_{2}+(n+1)a_{2}){z^2}+(ib_{1}+(n+1)a_{1}){z}+(ib_{0}+(n+1)a_{0})], z^{n+1} \rangle \\
& = (ib_{1}+(n+1)a_{1})\langle \Phi_{n},\Phi_{n} \rangle +(ib_{0}+(n+1)a_{0})\langle \Phi_{n},z^{n+1} \rangle 
\end{align*}
and using \eqref{fato1},
$ \langle A\Phi_{n}', z^n \rangle = 
(ib_{1}+(n+1)a_{1})\langle \Phi_{n},\Phi_{n} \rangle - (ib_{0}+(n+1)a_{0}) \overline{\gamma}_{n+1} \langle \Phi_{n},\Phi_{n} \rangle. $
Substituting the latter relation in \eqref{aux_antiga49}, we get
\begin{equation*}
\s_{n,n} = \s_{n,n-1} \overline{\gamma}_{n} \frac{\langle \Phi_{n-1},\Phi_{n-1} \rangle}{\langle \Phi_{n},\Phi_{n} \rangle}+ (ib_{1}+(n+1)a_{1}) - (ib_{0}+(n+1)a_{0}) \overline{\gamma}_{n+1}. 
\end{equation*}
Thus, with $\s_{n,n-1}$ given in \eqref{Bn_n-1}, the fact that $\kappa_{n-1}^{2}/\kappa_{n}^{2} =  1-|\alpha_{n-1}|^2$, and relation \eqref{lnn1} for $\gamma_{n+1}$, we get \eqref{antiga49}. 
\end{proof}

Since the reciprocal polynomials satisfy 
$\Phi_{n+1}^{*}(z) = \Phi_{n}^{*}(z) - \alpha_{n} z \Phi_{n}(z), $
then
$\Phi_{n+1}^{*}(z) + \alpha_{n} \Phi_{n+1}(z) = (1-|\alpha_{n} |^2) \Phi_{n}^{*}(z)$. Hence, we obtain a 
structure relation for the associated MOPUC, 
involving the polynomials $\Phi_{n+1}$, $\Phi_{n}$, $\Phi_{n-1}$ and $\Phi_{n+1}^{*}$.  
In the next corollary we recover  the result of Magnus \cite{Ma00} for $d=2$ and provide explicit formulas for the coefficients.
 
\begin{corollary}  
Under the hypothesis of Theorem \ref{prop_rel_est},
the associated MOPUC  satisfy 
\begin{align*} 
A(z)\Phi_{n}'(z) = & \
\s_{n,n}\Phi_{n}(z)
+ \s_{n,n-1}\Phi_{n-1}(z)
+ \frac{n a_2 + [ib_{2}-(n-1)a_{2}] |\alpha_n|^2}{1-|\alpha_n|^2}\Phi_{n+1}(z) \\
& + \frac{(ib_{2}+a_{2})\overline{\alpha}_{n}}{1-|\alpha_n|^2} \Phi_{n+1}^{\ast}(z), \quad  n \geqslant 2,
\end{align*}
with coefficients $\s_{n,n}$ and $\s_{n,n-1}$ given as in Theorem \ref{prop_rel_est}. 
\end{corollary}

Replacing $\s_{n,n}$ given by \eqref{antiga50} in the  structure relation \eqref{Est-geral2-nova}, we get 
\begin{align*}
A(z)\Phi_{n}'(z) =
& \ n a_2 \Phi_{n+1}(z) + [n a_1 - a_2 \gamma_{n} -na_{2}\overline{\alpha}_{n} \alpha_{n-1} ] \Phi_{n}(z) 
+ \s_{n,n-1}\Phi_{n-1}(z) \\[0.5ex]
& + (ib_2 +a_2) \overline{\alpha}_{n} \, [ \Phi_{n}^{\ast}(z) + \alpha_{n-1}  \Phi_{n}(z)].
\end{align*}
In the next result we provide a structure relation for the  MOPUC, involving the reciprocal polynomial $\Phi_{n-1}^{\ast}(z)$.

\begin{corollary} \label{coro2}
Under the hypothesis of Theorem \ref{prop_rel_est},
the associated  MOPUC satisfy 
\begin{align*}
A(z)\Phi_{n}'(z) = & \  n a_2 \Phi_{n+1}(z)+
[n a_1 - a_2 \gamma_{n}  -na_{2}\overline{\alpha}_{n} \alpha_{n-1} ] \Phi_{n}(z) 
+ [(ib_{0}+na_{0})(1-|\alpha_{n-1}|^2)]\Phi_{n-1}(z)
\nonumber \\[0.5ex]
& + (ib_2 +a_2) \overline{\alpha}_{n} (1-|\alpha_{n-1}|^2)
\Phi_{n-1}^{\ast}(z),\quad  n \geqslant 2.
\end{align*}
\end{corollary}

Using the results of Theorem \ref{prop_rel_est} we show in the following result that the Verblunsky coefficients satisfy some non-linear difference equations.

\begin{theorem} \label{TEO_eq_de_dif_caso_geral}
If the weight function satisfies \eqref{Eq-Tipo-Pearson-2} with $A(z)$ and $B(z)$ polynomials of degree at most 2 as in Theorem \ref{prop_rel_est}, then the Verblunsky coefficients satisfy the non-linear difference equations
\begin{equation} \label{eq_de_dif_caso_geral}
[(n-1)\overline{a}_2+i\overline{b}_2] \alpha_n 
+ [(n-1)\overline{a}_0-i\overline{b}_0]\alpha_{n-2}
= -(n \overline{a}_1 - \gamma_{n}\overline{a}_0 - \overline{\gamma}_{n} \overline{a}_2) \dfrac{\alpha_{n-1}}{1-|\alpha_{n-1}|^2}, \ n \geqslant 2,
\end{equation}
and
\begin{eqnarray} \label{eq_de_dif_caso_geral_2}
&& 
[(n-1)\overline{a}_2+i\overline{b}_2]\dfrac{\alpha_n}{1-|\alpha_{n-1}|^2}  
+ [(n-1) \overline{a}_0 - i \overline{b}_0]\alpha_{n-2} \\
&& \qquad  = - \big\{[i \overline{b}_0-(n+1) \overline{a}_0]\alpha_{n-1} \overline{\alpha}_{n} -i \overline{b}_1+(n+1) \overline{a}_1 - 2  \gamma_{n} \overline{a}_0  \big\} \dfrac{ \alpha_{n-1}}{1-|\alpha_{n-1}|^2}, \ n \geqslant 2, \nonumber
\end{eqnarray}
where $\alpha_n$ are the Verblunsky coefficients and $\gamma_{n}$ is given in \eqref{lnn1}.
\end{theorem}

\begin{proof}
Substituting $z=0$ in \eqref{Est-geral2-nova}, we get
\begin{equation} \label{com_eq_50}
a_0 \Phi_{n}'(0) = 
- \s_{n,n-1} \overline{\alpha}_{n-2}
- \s_{n,n} \overline{\alpha}_{n-1}
+ [ib_2 -(n-1)a_2] \overline{\alpha}_{n}.
\end{equation}
From \eqref{rnn1}, we know that
$ \Phi_{n}'(0)  
= - (1-|\alpha_{n-1}|^2)\overline{\alpha}_{n-2} - \overline{\alpha}_{n-1} \overline{\gamma}_{n}. $
Using this information, \eqref{Bn_n-1} and \eqref{antiga50} in equation \eqref{com_eq_50}, we have the following
\begin{align*}
- a_0 [(1-|\alpha_{n-1}|^2)\overline{\alpha}_{n-2} + \overline{\alpha}_{n-1} \overline{\gamma}_{n}]  = 
& \ [ib_2 -(n-1)a_2] \overline{\alpha}_{n}
-  (ib_{0}+na_{0})(1-|\alpha_{n-1}|^2) \overline{\alpha}_{n-2}  \\[0.5ex]
& - \{  [ib_{2}-(n-1)a_{2}]\overline{\alpha}_{n} \alpha_{n-1}  - a_2 \gamma_{n} + n a_1  \}\overline{\alpha}_{n-1},
\end{align*}
and after some simplifications we get \eqref{eq_de_dif_caso_geral}.

If we use \eqref{antiga49} instead of \eqref{antiga50}  in equation \eqref{com_eq_50}, we get the non-linear difference equation \eqref{eq_de_dif_caso_geral_2}. 
\end{proof}

\setcounter{equation}{0}
\section{Determination of semi-classical weight functions} 
\label{sec_Characterization}

The aim in this section is to determine all semi-classical weight functions, $w$, on the unit circle that satisfy \eqref{Eq-Tipo-Pearson-2}, for which  $A$ and $B$ are polynomials of degree at most 2.

Some semi-classical weight functions on the unit circle are already known, for example,  the Lebesgue weight function
$w(\theta) = 1/(2\pi)$ is semi-classical.
In Ismail's book \cite{Is05} some examples are presented. For instance, in \cite[p.~229]{Is05} and \cite{IsWi01} we find information about the circular Jacobi polynomials that are orthogonal with respect to 
\begin{align}  \label{CJP}
w(\theta) = \tau(\lambda) |1-e^{i\theta}|^{2\lambda} 
= \tilde{\tau}(\lambda) [\sin^2(\theta/2)]^\lambda,
\end{align}
where $\lambda>-1/2$, $\tau(\lambda)$ and $\tilde{\tau}(\lambda)$ are constants such that  $\mu_0 =1$. In \cite[p.~236]{Is05} and \cite{IsWi01} there is information about the modified Bessel polynomials that are orthogonal with respect to the semi-classical weight function
\begin{align}  \label{MBF}
w(\theta) = \frac{1}{2\pi I_0(t)} e^{t \cos(\theta)}, 
\end{align}
where $I_{\alpha}$ is the modified Bessel function. 

The orthogonal polynomials with respect to the semi-classical weight function
\begin{align}  \label{JOUC}
w(\theta) =   [\sin^{2}(\theta/2)]^{\lambda} \,
[\cos^{2}(\theta/2)]^{\beta} , 
\end{align}
with $\lambda > -1/2$ and $\beta > -1/2$, are called Jacobi polynomials on the unit circle, see 
 \cite{IsWi01} and \cite{Ma00}.  

In \cite{MaSR17}, in the context of coherent pairs of weight functions of the second kind 
 on the unit circle, the authors studied the following semi-classical weight functions. 
For $u \in \mathbb{C}$,  
 and $\tau(u)$ is a constant such that $\mu_0=1,$
\begin{align} \label{Teo3.1}
w(\theta)  =  \tau(u) \, e^{2|u| \sin(\theta + \arg(u))}.
\end{align}
For $u, r \in \mathbb{C},$ $ |r|\neq 1$, and $\tau(u,r)$ is a constant such that $\mu_0 =1,$
\begin{align} \label{Teo3.2}
w(\theta)  =  \tau(u, r) \, e^{2\Re(u/\overline{r}) \arg(1 -r e^{-i\theta})} |e^{i\theta}-r|^{-2\Im(u/\overline{r})}.
\end{align}
For $\lambda > -1/2$, $\eta \in \mathbb{R}$, $b = \lambda + i \eta,$  and $\tau(b)$ is a constant such that $\mu_0 =1,$
\begin{align} \label{Teo3.3}
w(\theta)  =  \tau(b) \, e^{- \eta \theta}
[\sin^{2}(\theta/2)]^{\lambda}.
\end{align}
The weight function \eqref{Teo3.3} was studied in Sri Ranga \cite{Ra10}.

Later we present the Pearson-type equation of the form \eqref{Eq-Tipo-Pearson-1} for each of these weight functions.

We continue using the notation $B(z)=b_{2}z^2+b_{1}z+b_{0}$ and since the coefficients $b_k$, $k=0,1,2$ are complex numbers, we denote $b_k=\Re(b_k)+i\Im(b_k),$
where $\Re(c)$ means real part of $c$ and $\Im(c)$ means imaginary part of   $c$.
Without loss of generality we consider $A$ a monic polynomial.  The analysis for the determination of semi-classical weight functions is divided into three cases, according to the zeros of the polynomial $A$.

\subsection{$A(z)$ is a polynomial of degree 0, $A(z)=1$}
\label{subsection_A(z)=1}

With $A(z)=1$,   from \eqref{Eq-Tipo-Pearson-3} we know that $w$ should satisfy
\begin{equation}\label{Medida-A-grau0}
\dfrac{dw(\theta)/d\theta}{w(\theta)}= b_2 z^2+ b_1 z + b_0, \quad  z=e^{i\theta}.
\end{equation}
In order to the weight function $w$ be positive in \eqref{Medida-A-grau0}, it is necessary that
$\Im(b_2 e^{2i\theta}+ b_1e^{i\theta} + b_0) =0,$
which is equivalent to
\begin{equation*}
\Im(b_2)\cos(2\theta)+\Re(b_2)\sin(2\theta) + \Re(b_1)\sin(\theta)+\Im(b_1)\cos(\theta) +\Im(b_{0})=0,
\end{equation*}
for $\theta \in [0,2\pi]$. Since the set $\{\cos(2\theta), \sin(2\theta), \cos(\theta), \sin(\theta), 1\}$ is linearly independent, then $b_{2}=0$, $b_{1}=0$ and $\Im(b_{0})=0$.
Hence, using the representation \eqref{Eq-Tipo-Pearson-4}, we get 
\begin{equation} \label{caso1-derivada-medida-4}
i\dfrac{dw(\theta)/dz}{w(\theta)}=\dfrac{\Re(b_0)}{z}, \quad  z=e^{i\theta}.
\end{equation}
Integrating \eqref{caso1-derivada-medida-4} with respect to $z$, it follows that
\begin{equation*}
\ln\left(\dfrac{w(\theta)}{w(\theta_0)}\right)=\Re(b_0)[\theta-\theta_0],
\end{equation*}
where $z_0 = e^{i \theta_0}$ with $\theta_0 \in [0,2\pi]$, such that $w(\theta_0) >0$.

Therefore, $ w(\theta)=w(\theta_0)e^{Re(b_0)[\theta-\theta_0]}. $ The additional condition to get the structure relation is $A(1)[w(2\pi) - w(0)]=0$.  We impose that $w(0)=w(2\pi)$, and hence $\Re(b_{0})=0$ and $w(\theta)=w(\theta_0)$. Since $w(\theta_0)$ is a constant, we can choose this constant such that $\int_{0}^{2\pi} w(\theta) d \theta =1$.
We conclude that the only semi-classical weight function on the unit circle with $A(z)=1$ is the Lebesgue weight function, namely, $w(\theta)= 1/(2\pi).$
Here the polynomial $B$ in the Pearson-type equation \eqref{Eq-Tipo-Pearson-1} is simply $B(z)=0$. We can observe that the Lebesgue weight function belongs to class $(0,0)$.

\subsection{$A(z)$ is a  polynomial of degree 1, $A(z)= z - r$, with $r \in \mathbb{C}$}
\label{section32}

From \eqref{Eq-Tipo-Pearson-3} with $A(z)= z-r$, $r \in \mathbb{C},$ we get
\begin{equation} \label{derivada-medida-A-grau1-caso2-a}
\frac{dw(\theta)/d\theta}{w(\theta)}= \frac{b_{2}z^2+(b_{1}-i)z+b_{0}}{z-r}, \quad  z=e^{i\theta}.
\end{equation}

In order to $w$ being positive, it is necessary  that the right hand side of  \eqref{derivada-medida-A-grau1-caso2-a} is real.
Observe that
\begin{equation*} 
\frac{dw(\theta)/d\theta}{w(\theta)} 
= \frac{-\overline{r}b_{2}z^2+[b_{2}-\overline{r}(b_{1}-i)]z-\overline{r}{b_{0}}+b_{1}-i+b_{0}z^{-1}}{|z-r|^2}
\end{equation*}
is real if and only if the numerator of the right hand side is real.
Hence, we assume that
\begin{equation*}
\Im(-\overline{r}b_{2}z^2+[b_{2}-\overline{r}(b_{1}-i)]z-\overline{r}{b_{0}}+b_{1}-i+b_{0}z^{-1})=0.
\end{equation*}
We use a process similar to the one adopted in Subsection \ref{subsection_A(z)=1}  to find the values of $b_k$, $k=0,1,2$, that satisfy the latter  assumption. Then,
we get five equations that can be written as a system of linear equations of order $5 \times 6$,
\begin{equation} \label{sys_1}
\textbf{F}_{1} \textbf{x}=\textbf{g}_{1}.
\end{equation}
The solution of this linear system is  
$\textbf{x}=(\Re(b_{2}),\Im(b_{2}),\Re(b_{1}),\Im(b_{1}),\Re(b_{0}),\Im(b_{0}))^{T},$  
the matrix $\textbf{F}_{1} $ and vector $\textbf{g}_{1}$ are
\begin{equation*}  
\textbf{F}_{1}=
\begin{pmatrix}
-\Re(r)  & -\Im(r) & 0 	& 		0	& 0      	& 0 \\
\Im(r)   & -\Re(r) & 0 	& 		0 	& 0      	& 0 \\
0        &  0      & 0 	& 		1 	& \Im(r) 	& -\Re(r)\\
1        & 0  	   & -\Re(r)& -\Im(r) & -1      & 0 \\
0        & 1       &  \Im(r)& -\Re(r) & 0       & 1
\end{pmatrix}, \quad
\textbf{g}_{1} = \begin{pmatrix}
0 \\
0 \\
1 \\
-\Im(r) \\
-\Re(r)
\end{pmatrix}.
\end{equation*}

Now we  consider separately the cases when $r=0$ and when $r \neq 0$.

\subsubsection{$A(z)= z$} 
\label{subsection321}

When $r= 0$, 
the solution of the linear system \eqref{sys_1}  yields $b_{2}=\overline{b}_0$ and $\Im(b_{1})=1$. Therefore,  using the representation \eqref{Eq-Tipo-Pearson-4} we have 
\begin{equation*}
i\dfrac{dw(\theta)/dz}{w(\theta)}=\dfrac{b_2z^2+\Re(b_1)z+\overline{b}_2}{z^2}, \quad z= e^{i\theta}.
\end{equation*}
Similarly, integrating with respect to $z$ and considering  $z_0 = e^{i \theta_0}, \theta_0 \in [0,2\pi]$, 
it follows that
\begin{equation*}
i\ln\left(\frac{w(\theta)}{w(\theta_{0})}\right) 
 = b_{2}(z-z_{0})-\overline{b}_{2} \overline{(z-z_{0})}+i\Re(b_1)(\theta-\theta_0).
\end{equation*}
Notice that
\begin{align*}
b_{2}(z-z_{0})-\overline{b}_2 \overline{(z-z_{0})} 
& = [e^{i(\theta-\theta_0)/2}-e^{-i(\theta-\theta_0)/2}][b_2e^{i(\theta+\theta_0)/2}+\overline{b}_2 e^{-i(\theta+\theta_0)/2}] \\[0.5ex]
& = 2i |b_2|[\sin(\theta+\arg(b_2))-\sin(\theta_0+\arg(b_2))].
\end{align*}
Thus
\begin{equation*}
\ln\left(\frac{w(\theta)}{w(\theta_{0})}\right) 
 = 2|b_2|[\sin(\theta+\arg(b_2))-\sin(\theta_0+\arg(b_2))]+\Re(b_1)(\theta-\theta_0)
\end{equation*}
and 
$w(\theta)=w(\theta_0)e^{ 2|b_2|[\sin(\theta+\arg(b_2))-\sin(\theta_0+\arg(b_2))]}e^{\Re(b_1)(\theta-\theta_0)}.$

Assuming that   $A(1)[w(2\pi) - w(0)]=0$, we have $\Re(b_{1})=0$. Hence, the weight function can be written as
\begin{equation} \label{medidaTeor31}
w(\theta)=\tau(b_{2}) \, e^{2|b_2|\sin(\theta+\arg(b_2))}.
\end{equation}
where $\tau(b_{2})=w(\theta_{0})$ $ e^{-2|b_2|[\sin(\theta_{0}+\arg(b_2)]}$.  This  weight function satisfies the Pearson-type equation \eqref{Eq-Tipo-Pearson-1} with $A(z)=z$ and $B(z)=b_{2}z^2+iz+\overline{b}_2.$  Considering $b_2 \neq 0$, we can see that the weight function \eqref{medidaTeor31} belongs to the semi-classical class $(1,2)$.

\begin{remark} \label{remark31}
The weight function \eqref{medidaTeor31} was studied in \cite{MaSR17} in the context of coherent pairs of weight functions of the second kind 
 on the unit circle, using $b_{2}=u \in \mathbb{C}$, see  \eqref{Teo3.1}. 
\end{remark}

\begin{remark} \label{remark32}
If we choose $b_{2}=it/2$ with $t>0$ and $\tau(b_{2})=1/[2\pi I_0(t)]$ in \eqref{medidaTeor31}, then  $\arg(b_{2})=\pi/2$ and the weight function is given by 
$ w(\theta)
= \tau(b_{2})  e^{t\cos(\theta)}, $
see \eqref{MBF}. This is the weight function related to the modified Bessel polynomials \cite{Is05}. 
\end{remark}

\subsubsection{$A(z)= z-r$, with $r\neq 0$}

The solution of system \eqref{sys_1} shows that $\Re(b_{2})=\Im(b_{2})=0$, i.e., $b_{2}=0$.

Now we analyse the solution of system \eqref{sys_1} when $|r| = 1$ and when $|r| \neq 1$.
\begin{description}
\item[1.] If $|r| = 1$, the solution of the system \eqref{sys_1} yields 
$b_{1} = -r\overline{b}_{0}+i $, where 
$b_{0}$  is an arbitrary complex number.
Then,   the representation \eqref{Eq-Tipo-Pearson-4} for this case becomes
\begin{align} \label{Eq-A-grau1-r-modulo1} 
i\dfrac{dw(\theta)/dz}{w(\theta)}=\dfrac{-r\overline{b}_0z+b_0}{z(z-r)}, \quad z=e^{i\theta}.
\end{align}
Again, integrating with respect to $z$ and considering  $z_0 = e^{i \theta_0}, \theta_0 \in [0,2\pi]$, it follows that
\begin{equation*} 
i\ln\left(\dfrac{w(\theta)}{w(\theta_0)}\right)  =-ib_0\overline{r}(\theta-\theta_0)+\ln\left(\dfrac{z-r}{z_0-r}\right)2\Im(b_0\overline{r})i.
\end{equation*}
Since $|r| = 1$, we denote $r=e^{i\varphi}$ and write
\begin{equation} \label{caso3.1.2b-f3}
\dfrac{z-r}{z_0-r} = e^{i(\theta-\theta_0)/2}\dfrac{\sin(\theta/2-\varphi/2)}{\sin(\theta_0/2-\varphi/2)}.
\end{equation}
Hence,
\begin{equation*}
\ln\left(\dfrac{w(\theta)}{w(\theta_0)}\right) 
= -\Re(b_{0}\overline{r})(\theta-\theta_{0}) + \ln\left[ \frac{\sin^2(\theta/2-\varphi/2)}{\sin^2(\theta_0/2-\varphi/2)}\right]^{\Im(b_0\overline{r})}.
\end{equation*}

Therefore, the weight function can be written as
\begin{equation*}
w(\theta)  = \tau(b_{0})e^{-\Re(\overline{b}_0{r})\theta}[ \sin^2(\theta/2-\varphi/2)]^{\Im(b_0\overline{r})},
\end{equation*}
with $\tau(b_{0})=w(\theta_{0})e^{\Re(\overline{b}_{0}{r})\theta_{0}}[\sin^2(\theta_{0}/2-\varphi/2)]^{-\Im(b_0\overline{r})}$. 

Imposing that  $A(1)[w(2\pi)-w(0)]=0$, we obtain two possible weight functions
\begin{enumerate}
\item[i)] If $\sin(\varphi/2)\neq 0$ (i.e., $r \neq 1$), it is necessary that $\Re(b_{0}\overline{r})=0,$ 
thus
\begin{equation*}
w(\theta)= \tau(b_{0})\left[\sin^2(\theta/2-\varphi/2)\right]^{\Im(b_0\overline{r})},  \quad \Im(b_0\overline{r}) > -1/2.
\end{equation*}
Here  $B(z)=(-r \overline{b}_{0}+i)z+b_{0}$, where $b_{0} \in \mathbb{R}$ when 
$\Re(r)=0$	 and $b_0 = ir\Im(b_{0})/\Re(r)$ when $\Re(r)\neq 0.$	

\item[ii)] If $\sin(\varphi/2)= 0$ (i.e., $r=1$), the weight function becomes
\begin{align}  \label{estaesta}
	w(\theta)= \tau(b_{0})e^{-\Re(b_0)\theta}[\sin^2(\theta/2)]^{\Im(b_0)}, \quad\Im(b_0)>-1/2.
\end{align}
In this case,  $B(z)= (-\overline{b}_{0}+i)z+b_{0}$.
\end{enumerate}

\begin{remark} \label{remark33}
 The weight function \eqref{estaesta} is the same as \eqref{Teo3.3},  studied in \cite{Ra10}, considering the parameter  $b_{0}=i\overline{b}$, where $b=\lambda + i \eta$, $\lambda > -1/2$ and $\eta \in \mathbb{R}.$ 
 This weight function satisfies the Pearson-type equation with $A(z) = z-1$ and $B(z) =  i [(b+1)z+\overline{b}]$.
\end{remark}

\item[2.] If $|r| \neq 1$, from the solution of the system \eqref{sys_1}, the coefficients 
$b_{1}$ and $b_{0}$  satisfy
\begin{equation*}
b_{1} = \begin{cases}
-\frac{\Re(b_{0})}{\Re(r)}+i, & \mbox{if } \Re(r)\neq 0, \\[0.5ex]
-\frac{\Im(b_{0})}{\Im(r)}+i, & \mbox{if } \Re(r)= 0,		\end{cases} 
\quad \mbox{and} \quad
b_{0} = \begin{cases}
\frac{\Re(b_{0})r}{\Re(r)}, & \mbox{if } \Re(r)\neq 0, \\[0.5ex]
i \Im(b_{0}), & \mbox{if } \Re(r)= 0,
\end{cases} 
\end{equation*}
where $\Re(b_{0})$ (resp.~$\Im(b_{0})$) is an arbitrary real number if $\Re(r) \neq 0$ (resp.~$\Re(r)=0$).

Hence, for $|r| \neq 1$, the representation \eqref{Eq-Tipo-Pearson-4} may be summarised as
\begin{equation*}
i\frac{dw(\theta)/dz}{w(\theta)}=\frac{s(r)}{z},
\quad \mbox{where} \quad
s(r)=\begin{cases}
-\frac{\Im(b_0)}{\Im(r)}, & \mbox{if } \Re(r)=0,\\[0.5ex]
-\frac{\Re(b_0)}{\Re(r)}, & \mbox{if } \Re(r)\neq 0.
\end{cases}
\end{equation*}
 Similar to what was done to get the equation \eqref{caso1-derivada-medida-4}, the weight function here satisfies
$w(\theta)=w(\theta_0)e^{s(r)[\theta-\theta_0]},$
and imposing that $w(0)=w(2\pi)$, we again get the Lebesgue weight function.   Here, the polynomials $A$ and $B$ in \eqref{Eq-Tipo-Pearson-1} are $A(z) = z-r$, with $|r| \neq 1$ and $B(z)=iz$.   Hence, we observe that the Lebesgue weight function also belongs to class $(1,1)$, as expected. 
\end{description}

\subsection{$A(z)$ is a polynomial of degree 2, $A(z)= (z-r_1)(z-r_2)$ with $r_1, r_2 \in \mathbb{C}$}

Taking into account that $A(z)= (z-r_1)(z-r_2),$
the expression \eqref{Eq-Tipo-Pearson-3} is equivalent to
\begin{equation*}
\frac{dw(\theta)/d\theta}{w(\theta)} = \frac{[(b_{2}-2i)z^2+[b_{1}+i(r_1+r_2)]z+b_{0}][(\overline{z}-\overline{r}_1)(\overline{z}-\overline{r}_2)]}{|z-r_1|^2|z-r_2|^2}.
\end{equation*}
Again, the weight function $w$ is positive if
\begin{equation} \label{esteseste}
 \Im{\{[(b_{2}-2i)z^2+(b_{1}+i(r_1+r_2))z+b_{0}][(\overline{z}-\overline{r}_1)(\overline{z}-\overline{r}_2)]\}}=0,
\end{equation}
for $z= e^{i\theta}$. Similar to what was done in Section \ref{section32}, to find the values of $b_k$, $k=0,1,2,$ that satisfy \eqref{esteseste} is equivalent to solve the system of linear equations of order $5 \times 6$,
\begin{align} \label{Sistema-A-grau2}
\textbf{F}_{2}\textbf{x}=\textbf{g}_{2},
\end{align}
where  $\textbf{x}=(\Im(b_{0}),\Re(b_{0}),\Im(b_{1}),\Re(b_{1}),\Im(b_{2}),\Re(b_{2}))^{T}$,
\begin{equation*}
\textbf{F}_{2}=
\begin{pmatrix}
1&0&0&0&\Re(r_1r_2)&-\Im(r_1r_2)\\
0&-1&0&0&\Im(r_1r_2)&\Re(r_1r_2)\\
\Re(r_1r_2)&-\Im(r_1r_2)&-\Re(r_1+r_2)&\Im(r_1+r_2)&1&0\\
-\Re(r_1+r_2)&\Im(r_1+r_2)&\Re(r_1r_2)+1&-\Im(r_1r_2)&-\Re(r_1+r_2)&\Im(r_1+r_2)\\
\Im(r_1+r_2)&\Re(r_1+r_2)&\Im(r_1r_2)&\Re(r_1r_2)-1&-\Im(r_1+r_2)&-\Re(r_1+r_2)
\end{pmatrix}
\end{equation*}
and
\begin{equation*}
\textbf{g}_{2}= 
\begin{pmatrix}
2\Re(r_1r_2)\\
2\Im(r_1r_2)\\
2+|r_1+r_2|^2\\
-3\Re(r_1+r_2)-|r_1|^2\Re(r_2)-|r_2|^2\Re(r_1)\\
-3\Im(r_1+r_2)-|r_1|^2\Im(r_2)-|r_2|^2\Im(r_1)
\end{pmatrix}.
\end{equation*} 

It is easier to determine the solution of the linear system \eqref{Sistema-A-grau2} if we consider cases according with the zeros $r_1$ and $r_2$.

\subsubsection{$A(z)= (z-r_1)(z-r_2)$ with $|r_1| = 1$ and $|r_2| = 1$}

We split the analysis into two cases  $r_1 = r_2=r$ or $r_1 \neq r_2$.

\begin{description}
\item[1.] For $A(z)= (z-r)^2 $, with $|r|= 1$,  the solution of the linear system \eqref{Sistema-A-grau2} is
\begin{equation*}
b_{2}= r^2\overline{b}_0+2i, \quad   \quad 
b_{1}=
\begin{cases}
 \frac{2}{\Im(r)}+\Im(b_1)i,  & \mbox{if} \  \Re(r)=0, \\[0.8ex]
 \frac{\Re(b_1)r-2i}{\Re(r)}, & \mbox{if} \ \Re(r)\neq 0,  
\end{cases} 
\end{equation*}
where $b_{0}$ is an arbitrary complex number and $\Re(b_{1})$ (resp.~$\Im(b_{1})$) is  an arbitrary real number if $\Re(r) \neq 0$ (resp.~$\Re(r)=0$).

Hence, the Pearson-type equation \eqref{Eq-Tipo-Pearson-4} becomes
\begin{equation*}
i\dfrac{dw(\theta)/dz}{w(\theta)}=\dfrac{r^2\overline{b}_0 z^2+(b_1 + 2ri)z+b_0}{z(z-r)^2}, \quad z=e^{i\theta}.
\end{equation*}
Integrating with respect to $z=e^{i\theta}$ 
and considering $z_{0}=e^{i\theta_{0}}$ with $\theta_{0} \in [0,2\pi]$, we obtain
\begin{align} \nonumber
i\ln\left(\dfrac{w(\theta)}{w(\theta_0)}\right) & =[2\Re(r^2\overline{b}_0)+(b_1 + 2ri)\overline{r}]\dfrac{r(z-z_0)}{(z-r)(z_0-r)} \\ \label{A-grau2-raizes-iguais} 
& \ \ \   +ib_0\overline{r}^2  (\theta-\theta_0) +  2i\Im(r^2\overline{b}_0)\ln\left(\dfrac{z-r}{z_0-r}\right).  
\end{align}
Observe that $(b_1 + 2ri )\overline{r} \in \mathbb{R}$ and
\begin{equation*} 
(b_1 + 2ri)\overline{r}=
\begin{cases}
\Im(b_1)\Im(r), & \mbox{ se } \Re(r)=0,\\[0.5ex]
\dfrac{\Re(b_1)-2\Im(r)}{\Re(r)}, & \mbox{ se } \Re(r)\neq 0. 
\end{cases} 
\end{equation*}

Since $|r|=1$ we denote $r=e^{i\varphi}$, with $0\leq \varphi \leq 2\pi$. Thus, from
\begin{equation*} 
\dfrac{r(z-z_0)}{(z-r)(z_0-r)} =\dfrac{\sin(\theta/2-\theta_0/2)}{2i\sin(\theta/2-\varphi/2)\sin(\theta_0/2-\varphi/2)} =\frac{\cot(\theta_0/2-\varphi/2)-\cot(\theta/2-\varphi/2)}{2i}
\end{equation*}
and \eqref{caso3.1.2b-f3}, the equation \eqref{A-grau2-raizes-iguais} becomes
\begin{align*}
\ln\left(\dfrac{w(\theta)}{w(\theta_0)}\right)
= & -\left[\Re(r^2\overline{b}_{0})+\frac{(b_1+2ri)\overline{r}}{2}\right][\cot(\theta_0/2-\varphi/2)-\cot(\theta/2-\varphi/2)]  \\
  & + \Re(\overline{b}_{0}{r}^2)(\theta-\theta_{0})
  +\ln\left(\dfrac{\sin^2(\theta/2-\varphi/2)}{\sin^2(\theta_0/2-\varphi/2)}\right)^{\Im(r^2\overline{b}_{0})}.
\end{align*}
The corresponding weight function is 
\begin{equation*}   
w(\theta)  = \tau(b_{0};b_{1})\left[\sin^2(\theta/2-\varphi/2)\right]^{2\Im(\overline{b}_0{r}^2)}e^{\Re(\overline{b}_0{r}^2)\theta} e^{[\Re(\overline{b}_0{r}^2)+\frac{(b_1+2ri)\overline{r}}{2}][\cot(\theta/2-\varphi/2)]},
\end{equation*}
where $\tau(b_{0};b_{1})$ is a constant.

Assuming that $A(1)[w(2\pi)-w(0)]=0$, we obtain two possible weight functions
\begin{enumerate}
\item[i)]  If $\sin(\varphi/2)\neq 0$, it is necessary that  $\Re(\overline{b}_0{r}^2)=0$ and  $(b_1+2ri)\overline{r}=0$, then the weight function is
\begin{equation*} 
w(\theta)=\tau(b_{0};b_{1})\left[\sin^2(\theta/2-\varphi/2)\right]^{\Im(\overline{b}_0{r}^2)},  \quad \Im(\overline{b}_0{r}^2)>-1/2.
\end{equation*}
In this case,  $B(z)=(r^2\overline{b}_{0}+2i)z^2-(2ir)z+b_{0}$,  where $b_{0} \in \mathbb{R}$ when 
$\Re(r^2)=0$ and $b_0 =  ir^2\Im(b_{0})/\Re(r^2)$ when $\Re(r^2)\neq 0.$	

\item[ii)] If  $\sin(\varphi/2)= 0$, it is necessary that $\Re(\overline{b}_0)+(b_1+2i)/2=0$, and
\begin{equation}  \label{estaesta2}
w(\theta)=\tau(b_{0};b_{1})
e^{\Re(\overline{b}_0)\theta}
\left[\sin^2(\theta/2)\right]^{\Im(\overline{b}_0)}, \quad \Im(\overline{b}_{0})>-1/2, \quad \Re(\overline{b}_0)\in\mathbb{R}.
\end{equation} 
Here,  $B(z)=(\overline{b}_{0}+2i)z^2 +[-(b_{0}+\overline{b}_{0})-2i]z+b_{0}$.
\end{enumerate}

\begin{remark} \label{remark34}
As in Remark \ref{remark33}, we conclude that the weight function \eqref{estaesta2} is the  weight function \eqref{Teo3.3} considering the parameter  $b_{0}=-i\overline{b}$,  where $b=\lambda + i \eta$, $\lambda > -1/2$ and $\eta \in \mathbb{R}.$  
In this case, the polynomials  in \eqref{Eq-Tipo-Pearson-1} are $A(z)= (z-1)^2$ and  $B(z)=i [({b}+2)z^2+(\overline{b}-b-2)z-\overline{b}].$
\end{remark}

\item[2.] For $A(z)=(z-r_1)(z-r_2)$, $|r_1|= |r_2|=1$ and $r_1 \neq r_2$, to find the solution of the linear system
\eqref{Sistema-A-grau2}, we need to divide the choices for the values of $r_1$ and $r_2$ into two cases,
$\Re(r_1+r_2)=0$ and  $\Re(r_1+r_2)\neq 0$.
When $\Re(r_1+r_2)=0$, we get more two possibilities 
 $r_2=-r_1$ or $r_2 =-\overline{r}_1$. 
For the case $r_2 =-r_1$,  we also get different solutions of the system if $\Re(r_1)^2=1$ or $\Re(r_1)^2 \neq 1$.
In summary, we have 			
\begin{align*}
& b_{2}=2i+r_1 r_2\overline{b}_{0}, \\
& b_{1}=\begin{cases}
     \Im(b_1) i, & \mbox{if } r_2=-r_1 \mbox{ and } \Re(r_1)^2=1,\\[0.5ex]
     -\frac{r_1 \Re(b_{1}) i}{\Im(r_1)}, & \mbox{if } r_2=-r_1 \mbox{ and } \Re(r_1)^2\neq  1, \\[0.5ex]
		2\Im(r_1) + \Im(b_{1}) i, & \mbox{if } r_2=-\overline{r}_1 \mbox{ and } \Re(r_1)\Im(r_1)\neq 0, \\[0.5ex]
		\frac{\Re(b_{1})(r_1+r_2)-|r_1+r_2|^2 i}{\Re(r_1+r_2)}, & \mbox{if } \Re(r_1+r_2)\neq 0,
\end{cases}
\end{align*}
where  $b_{0}$  is an arbitrary complex number.
In addition, either $\Re(b_1)$ or   $\Im(b_1)$  is an arbitrary depending on which subcase are $r_1$ and $r_2$. 
Hence, the representation \eqref{Eq-Tipo-Pearson-4} becomes
\begin{equation*} 
i\frac{d w(\theta)/d z}{w(\theta)} = \frac{r_1r_2 \overline{b}_0z^2+[b_{1}+i(r_1+r_2)]z+b_{0}}{z(z-r_1)(z-r_2)}.
\end{equation*}
Using similar integration as before, with $z_0 = e^{i\theta_0}, \theta_0 \in [0,2\pi]$, we get
\begin{equation}\label{A-grau2-r1r2-modulo1}
i \ln\left(\frac{w(\theta)}{w(\theta_{0})}\right) = ib_{0}\overline{r}_1 \overline{r}_2(\theta-\theta_{0}) + H_{1}\ln\left(\frac{z-r_1}{z_{0}-r_1}\right)+H_{2}\ln\left(\frac{z-r_2}{z_{0}-r_2}\right),
\end{equation}
where
\begin{align*}
& H_{1}=-\frac{1}{r_2-r_1}[r_1^2r_2\overline{b}_{0}+b_{1}+i(r_1+r_2)+\overline{r}_1b_{0}], \\
& H_{2}=\frac{1}{r_2-r_1}[r_2^2r_1\overline{b}_{0}+b_{1}+i(r_1+r_2)+\overline{r}_2b_{0}].
\end{align*}			
Notice that $H_{1}+H_{2}=2i\Im(\overline{b}_0r_1r_2)$ and thus we can rewrite  	\eqref{A-grau2-r1r2-modulo1} as 	
\begin{align} 
\ln\left(\frac{w(\theta)}{w(\theta_{0})}\right)   =  & \ b_{0}\overline{r}_1 \overline{r}_2(\theta-\theta_{0})+2\Im(\overline{b}_0r_1r_2)\ln\left(\frac{z-r_2}{z_{0}-r_2}\right) \nonumber \\
& +iH_{1}\left[
\ln\left(\frac{z-r_2}{z_{0}-r_2}\right)
- \ln\left(\frac{z-r_1}{z_{0}-r_1}\right) 
\right].
\label{A-grau2-r1r2-modulo1=denovodenovo}
\end{align}			
On the other hand  $H_{1} $ can be written as
\begin{equation*} 
H_{1}  =-\frac{1}{|r_2-r_1|^2}[2i\Im[\overline{b}_0r_1(r_1-r_2)]-2\Im(r_1\overline{r}_2)+b_{1}(\overline{r}_2-\overline{r}_1)].
\end{equation*}

Now using the corresponding values of $b_{1}$ for each solution of the system, we summarise the values of  $H_{1}$ as
\begin{align*}
\begin{cases}
[\Im(b_{0})+\frac{\Im(b_{1})r_1}{2}]i, 
\hspace*{3cm} \mbox{if } r_2=-r_1  \mbox{ and } \Re(r_1)^2=1,\\[0.5ex]
-\left[2\Im(\overline{b}_0 r_1^2) +\frac{\Re(b_{1})}{\Im(r_1)}\right] \frac{i}{2}, 
\hspace*{2cm}  \mbox{if } r_2=-r_1  \mbox{ and } \Re(r_1)^2\neq 1,\\[0.5ex]
 \frac{\left[2\Im(\overline{r}_1b_{0})+\Im(b_{1})\right] i}{2\Re(r_1)}, 
 \hspace*{3.3cm}  \mbox{if } r_2=-\overline{r}_1
 \mbox{ and } \Re(r_1)\Im(r_1) \neq 0,  \\[0.5ex]
\{-\Im[\overline{b}_0r_1(r_1-r_2)]+\frac{\Im(\overline{r}_1 r_2)]}{\Re(r_1+r_2)}[\Re(b_{1})-\Im(r_1+r_2)] \}\frac{2i}{|r_2-r_1|^2}, 
\hspace*{0.5cm}  \mbox{if } \Re(r_1+r_2)\neq 0.
\end{cases}
\end{align*}	

Since $|r_1|=|r_2|=1$ and $r_1\neq r_2$, from the above values for $H_{1}$ one can see that $\Re(H_{1})=0$.
Thus, equation \eqref{A-grau2-r1r2-modulo1=denovodenovo} can be written as
\begin{equation*}  
\ln\left(\frac{w(\theta)}{w(\theta_{0})}\right)   =  b_{0}\overline{r}_1 \overline{r}_2(\theta-\theta_{0})+2\Im(\overline{b}_0r_1r_2)\ln\left(\frac{z-r_2}{z_{0}-r_2}\right) +\Im(H_{1})\left[\ln\left(\frac{z-r_1}{z_{0}-r_1}\right)-\ln\left(\frac{z-r_2}{z_{0}-r_2}\right)\right].
\end{equation*}					
	
Denoting $r_1=e^{i\varphi}$ and $r_2=e^{i\phi}$ as in \eqref{caso3.1.2b-f3}, we have
\begin{equation*} 
\frac{z-r_1}{z_{0}-r_1}  = e^{i(\theta-\theta_{0})/2} \frac{\sin[(\theta-\varphi)/2]}{\sin[(\theta_{0}-\varphi)/2]} \quad \mbox{and}  \quad 
\frac{z-r_2}{z_{0}-r_2}= e^{i(\theta-\theta_{0})/2} \frac{\sin[(\theta-\phi)/2]}{\sin[(\theta_{0}-\phi)/2]}.
\end{equation*}					
Hence, it follows that  
\begin{align*}
\ln\left(\frac{w(\theta)}{w(\theta_{0})}\right) = 
& \Re(b_{0}\overline{r}_1\overline{r}_2)(\theta-\theta_{0}) + \ln\left|\frac{\sin[(\theta-\phi)/2]}{\sin[(\theta_{0}-\phi)/2]}\right|^{-2\Im(b_{0}\overline{r}_1\overline{r}_2)} \\
& + \ln\left|\frac{\sin[(\theta-\varphi)/2]}{\sin[(\theta_{0}-\varphi)/2]}\right|^{\Im(H_{1})} + \ln\left|\frac{\sin[(\theta-\phi)/2]}{\sin[(\theta_{0}-\phi)/2]}\right|^{-\Im(H_{1})},
\end{align*}					
and the weight function becomes
\begin{equation*} 
w(\theta) = \tau(b_0;b_1) e^{\Re(b_{0}\overline{r}_1\overline{r}_2)\theta}\left[\sin^2(\theta/2-\phi/2)\right]^{-[\frac{1}{2}\Im(H_{1})-\Im(\overline{b}_{0}{r_1r_2})]} \left[\sin^2(\theta/2-\varphi/2)\right]^{\frac{1}{2}\Im(H_{1})},
\end{equation*}				
with $\tau(b_0;b_1)=e^{-\Re(\overline{b}_{0}{r_1r_2})\theta_{0}}\left[\sin^2(\theta_{0}/2-\phi/2)\right]^{[\frac{1}{2}\Im(H_{1})-\Im(\overline{b}_{0}{r_1r_2})]} \left[\sin^2(\theta_{0}/2-\varphi/2)\right]^{-\frac{1}{2}\Im(H_{1})}$.

Assuming that $A(1)[w(2\pi)-w(0)]=0$, we obtain two possible weight functions
\begin{enumerate}
\item[i)] If $\sin(\varphi/2)\sin(\phi/2)\neq 0$, 
(i.e., $r_1 \neq 1$, $r_2 \neq 1$), 
it is necessary that $\Re(\overline{b}_{0}r_1r_2)=0$ and the weight function becomes 
\begin{equation*}
w(\theta)= \tau(b_0;b_1) \left[\sin^2(\theta/2-\phi/2)\right]^{-[\frac{1}{2}\Im(H_{1})-\Im(\overline{b}_{0}{r_1 r_2})]} \left[\sin^2(\theta/2-\varphi/2)\right]^{\frac{1}{2}\Im(H_{1})}.
\end{equation*}
Here, $B(z)=(\overline{b}_{0}r_1r_2+2i)z^2+b_{1}z+b_{0}$, where 
$b_{0}$ satisfies $\Re(\overline{b}_{0}{r_1r_2})=0$.
	
\item[ii)] Consider $\sin(\varphi/2)\sin(\phi/2)= 0$.
(i.e., $(r_1-1)(r_2-1)=0$). 
First we suppose $r_2=1$, hence $\phi=0$ and 
\begin{equation*}
w(\theta)=\tau(b_0,b_1) e^{\Re(\overline{b}_{0}{r_1})\theta}\left[\sin^2(\theta/2)\right]^{-[\frac{1}{2}\Im(H_{1})-\Im(\overline{b}_{0}{r_1})]} \left[\sin^2(\theta/2-\varphi/2)\right]^{\frac{1}{2}\Im(H_{1})},
\end{equation*}
with  $\Im(H_{1})/2-\Im(\overline{b}_{0}{r_1}) < 1/2$ and $\Im(H_{1})/2 > -1/2$.

Now choosing $r_1=-1$, 
the weight function becomes
\begin{equation*}
 w(\theta)=\tau(b_{1},b_{0}) e^{-\Re(b_{0})\theta}\left[\sin^2(\theta/2)\right]^{\frac{1}{2}\Im({b}_{0})+\frac{1}{4}\Im(b_{1})} \left[\cos^2(\theta/2)\right]^{{\frac{1}{2}\Im({b}_{0})-\frac{1}{4}\Im(b_{1})}},
\end{equation*}
where $ \Im(b_0)/2+\Im(b_1)/4 > -1/2 $ and  $ \Im(b_0)/2- \Im(b_1)/4 > -1/2$.
In this case  $B(z)=(-\overline{b}_0 +2i)z^2 +i \Im(b_1)z + b_0$.

\begin{remark}  \label{remark35}
Denoting $\eta = \Re(b_0) \in \mathbb{R},$ $ \lambda=  \Im(b_0)/2+ \Im(b_1)/4$ and $\beta =  \Im(b_0)/2-\Im(b_1)/4$, hence
$b_{0}=\eta+i(\lambda+\beta)$ and $\Im(b_{1})=2(\lambda-\beta)$, then
we get a new semi-classical weight function on the unit circle
\begin{equation} \label{newfunction}
w(\theta)=\tau(\lambda,\beta,\eta) \, e^{- \eta \, \theta }\, [\sin^2(\theta/2)]^{\lambda}[\cos^2(\theta/2)]^{\beta},
\end{equation}
where $\tau(\lambda,\beta,\eta)$ is a constant, $\eta \in \mathbb{R}$,  $\lambda >-1/2$, $\beta >-1/2$, and $0 \leqslant \theta \leqslant 2\pi$. 
 This weight function satisfies the Pearson-type equation \eqref{Eq-Tipo-Pearson-1} with $A(z)=z^2-1$ and $B(z)
= i[(\lambda+\beta+i\eta+2)z^{2}+2(\lambda-\beta)z + \lambda+\beta-i\eta]$. Hence, the weight function \eqref{newfunction} belongs to the class $(2,2)$.
\end{remark}

\begin{remark} \label{remark36}
We observe that the semi-classical weight function \eqref{newfunction}  is an extension of other semi-classical weight functions, since

- when $\eta=0$ this weight function is  \eqref{JOUC}, known as the Jacobi weight function on the unit circle,  see \cite{Ma00}. 

- when $\eta=0$ and $\beta=0$ this weight function is  \eqref{CJP}, and the associated orthogonal polynomials are known as circular Jacobi polynomials, see \cite{Is05}.

- when $\beta=0$ this weight function is  \eqref{Teo3.3},  studied in \cite{Ra10}. 
\end{remark}

\noindent {\bf Special case}

Considering the weight function \eqref{newfunction} as
\begin{equation*}
\tilde{w}(t)  =  \tau(\beta, \eta) \, e^{- \eta \, t}
[\sin^{2}(t/2)]^{\beta}, \quad t \in [0,2\pi],
\end{equation*}
setting $\theta = t - \pi$
and  the rotation $w(\theta) = \tilde{w}(\theta+\pi)$, we obtain
\begin{equation} \label{exp_cos}
w(\theta)=\tau(\beta,\eta) \, e^{- \eta \, \theta}\, [\cos^2(\theta/2)]^{\beta}, \quad \theta \in[-\pi,\pi]
\end{equation}
that vanishes for $\theta = \pm \pi$. Hence the inner product becomes
\begin{equation*}
\langle f, g \rangle = \tau(\beta,\eta) \int_{-\pi}^{\pi} f(e^{i\theta}) \overline{g(e^{i\theta})} \, e^{-\theta \, \eta}\, [\cos^2(\theta/2)]^{\beta} d\theta.
\end{equation*}
This weight function satisfies the Pearson-type equation \eqref{Eq-Tipo-Pearson-1} with $A(z)=z^2-1$.
\end{enumerate}
\end{description}

\subsubsection{$A(z)= (z-r_1)(z-r_2)$ with $|r_1| \neq 1$ and $|r_2| = 1$}

We split the analysis into two cases  $r_1 \neq 0 $ and $r_1 =0$.

\begin{description}
\item[1.] For $A(z)= (z-r_1)(z-r_2)$ with $|r_1 |\neq 1$, $r_1 \neq 0$, and $|r_2|=1$,  the solution of the linear system \eqref{Sistema-A-grau2} is
\begin{equation*} 
b_{2} = \frac{\overline{b}_{0} r_2}{\overline{r}_1} + 2i \quad \mbox{and} \quad
b_{1}= -\frac{1}{|r_1|^2}(r_1^2r_2 \overline{b}_{0}+\overline{r}_1b_{0}) -(r_1+r_2)i
\end{equation*}
with $b_{0}$ an arbitrary complex number.

Substituting this solution in \eqref{Eq-Tipo-Pearson-4} we obtain
\begin{equation*} 
i\frac{dw(\theta)/dz}{w(\theta)}= \frac{\overline{b}_{0}r_1 r_2 z-b_{0}\overline{r}_1}{|r_1|^2z(z-r_2)}.
\end{equation*}
Similar to what was done in \eqref{Eq-A-grau1-r-modulo1}, considering $r_2=e^{i\phi}$, we obtain the weight function
\begin{equation*} 
w(\theta) = \tau(b_{0})e^{\Re\left(\frac{\overline{b}_{0} r_2}{\overline{r}_1}\right) 
\theta}\left[\sin^2(\theta/2-\phi/2)\right]^{
\Im\left(\frac{\overline{b}_{0} r_2}{\overline{r}_1}\right) 
},
\end{equation*}
where $\tau(b_{0})$ is a constant.

Imposing that  $A(1)[w(2\pi)-w(0)]=0$, we obtain two possible weight functions

\begin{enumerate}
\item[i)] If $\sin(\phi/2) \neq 0$ (i.e., $r_2 \neq 1$), it is necessary that   $\Re(\overline{b}_{0}{r_1 r_2})=0$ and the weight function becomes
\begin{equation*} 
w(\theta) =\tau(b_{0})\left[\sin^2(\theta/2-\phi/2)\right]^{\Im\left(\frac{\overline{b}_{0} r_2}{\overline{r}_1}\right)}.
\end{equation*}
In this case  $B(z)=\left(\frac{\overline{b}_{0} r_2}{\overline{r}_1} + 2i\right)z^2-
\left(  \frac{\overline{b}_{0} r_1 r_2}{\overline{r}_1}+\frac{b_{0} }{r_1}  +(r_1+r_2)i \right) z
+b_{0}$. We observe that the condition $\Re(\overline{b}_{0}{r_1 r_2})=0$ is equivalent to
$b_0 \in \mathbb{R}$ if $\Re(r_1 r_2)=0$ or $b_0 = i\Im(b_{0})r_1 r_2/\Re(r_1 r_2)$ if $\Re(r_1 r_2)\neq 0$.

\item[ii)] If $\sin(\phi/2)=0$ (i.e., $r_2=1$), 
then
\begin{align} \label{estaesta4}
w(\theta)=\tau(b_{0}) e^{\Re\left(\frac{{b}_{0}}{{r}_1}\right) \theta}
\left[\sin^2(\theta/2)\right]^{-\Im\left(\frac{{b}_{0}}{{r}_1}\right)},
\end{align}
with  $\Im({b}_{0}/{r}_1) < 1/2.$

\begin{remark}   \label{remark37}
As in Remarks \ref{remark33} and \ref{remark34}, we see that the weight function \eqref{estaesta4} is the  weight function \eqref{Teo3.3} considering $r=r_1 \in \mathbb{C}$,  
$b_{0}=-ir \overline{b}$, $ b= \lambda + i \eta$, $\lambda =  -\Im({b}_{0}/{r})$, and $\eta = - \Re({b}_{0}/ {r})$. This weight function satisfies the Pearson-type equation \eqref{Eq-Tipo-Pearson-1} with polynomials $A(z)=(z-1)(z-r)$ and $B(z)= i\{(b+2)z^2+[\overline{b}-1-r(b+1)]z-{r}\overline{b}\}$.
\end{remark}

\end{enumerate}

\item[2.] If $r_1 =0$ and $|r_2|=1$, setting $r=r_2$ we can write   $A(z)=z(z-r)$ with $|r|=1$,  and the solution of the linear system \eqref{Sistema-A-grau2} is
$b_{2}= -  \overline{b}_1 r +3i $ and $b_{0}= 0$
with $b_{1}$ an arbitrary complex number.

Using this solution the  representation \eqref{Eq-Tipo-Pearson-4} becomes
\begin{equation*}  
i\frac{dw(\theta)/dz}{w(\theta)} 
= \frac{-r(\overline{b}_1-i\,\overline{r})z+(b_1+ir)}{z(z-r)}, \quad z=e^{i\theta}.
\end{equation*}
Analogously to what was done in \eqref{Eq-A-grau1-r-modulo1}, considering $r=e^{i\varphi}$, it follows that
\begin{equation*}
w(\theta) 
= \tau (b_{1})e^{-\Re(\overline{b}_1{r}-i)\theta}\left[\sin^2(\theta/2-\varphi/2)\right]^{-\Im(\overline{b}_1{r}-i)},
\end{equation*}
with $\Im(\overline{b}_1r-i)<1/2$ and
$\tau (b_{1}) = e^{\Re(\overline{b}_1{r}-i)\theta_{0}}\left[\sin^2(\theta_{0}/2-\varphi/2)\right]^{\Im(\overline{b}_1{r}-i)}$.

Assuming that  $A(1)[w(2\pi)-w(0)]=0$, we obtain two possible weight functions
\begin{enumerate}
\item[i)] If $\sin(\varphi/2)\neq 0$, 
it is necessary that  $\Re(\overline{b}_{1}{r})=0$ and we obtain
\begin{equation*}
	w(\theta)=\tau (b_{1})\left[\sin^2(\theta/2-\varphi/2)\right]^{-\Im(\overline{b}_1{r}-i)}, 
\end{equation*}
with $\Im(\overline{b}_1r-i)<1/2.$	
 Here, $A(z)=z(z-r)$,  $B(z)=(3i-\overline{b}_1 r)z^2+b_{1}z$, and $\Re(\overline{b}_{1}r)=0$  is equivalent to
$ b_1 \in \mathbb{R}$ if $\Re(r)=0$  or $b_1 = i\Im(b_{1})r/\Re(r)$ if $\Re(r)\neq 0$.

\item[ii)] If $\sin(\varphi/2)= 0$, 
then
\begin{equation*}
w(\theta)=\tau (b_{1})e^{- \Re(\overline{b}_1-i)\, \theta}\left[\sin^2(\theta/2)\right]^{ -\Im(\overline{b}_1-i) },
\end{equation*}
with $ \Im(\overline{b}_1-i) < 1/2.$

\begin{remark}  \label{remark38}
Once more we obtain the  weight function \eqref{Teo3.3} considering  $b_{1}=i(\lambda-i\eta)-i$, 
$\lambda = -\Im(\overline{b}_1-i),$  and
$ \eta = 	\Re(\overline{b}_1-i).$
 This weight function satisfies the Pearson-type equation \eqref{Eq-Tipo-Pearson-1} with  $A(z)=z(z-1)$ and $ B(z)=i[(b+2)z^2+(\overline{b}-1)z]$, for $ b= \lambda + i \eta.$  
\end{remark}
\end{enumerate}
\end{description}

\subsubsection{$A(z)= (z-r_1)(z-r_2)$ with $|r_1| \neq 1$ and $|r_1r_2| = 1$}
\label{subsection334}

Since we are considering $|r_1| \neq 1$, $|r_2| \neq 1,$ and  $|r_2|=1/|r_1|$, we find the solutions of linear system \eqref{Sistema-A-grau2}
for two possibilities
$r_2= 1/\overline{r}_1$ and $r_2 \neq 1/\overline{r}_1$.

\begin{description}
\item[1.] First we consider  $|r_1| \neq 1$ and  $r_2=1/\overline{r}_1$. 
Denoting $r_1=r$ and 
$r_2=1/\overline{r}$, we get $A(z)= (z-r)(z-1/\overline{r})$. In this case the solution of  \eqref{Sistema-A-grau2} is given by 
\begin{equation*}
 b_{2}=\frac{\overline{b}_0r}{\overline{r}} + 2i \quad 
 \mbox{and} \quad
 b_{1}=\begin{cases} 
\frac{\Re(b_{1})r-(|r^2|+1)i}{\Re(r)}, & \mbox{if } \Re(r)\neq 0,\\[0.5ex]   
\frac{|r^2|+1}{\Im(r)}+\Im(b_{1})i, & \mbox{if } \Re(r)= 0,
\end{cases} 
\end{equation*}
where $b_{0}$ is an arbitrary complex number, and either $\Re(b_{1})$ or $\Im(b_{1})$ is an arbitrary real number depending on whether $\Re(r)=0$ or $\Re(r) \neq 0$.

The representation \eqref{Eq-Tipo-Pearson-4} becomes
\begin{equation} \label{frfrfrfr}
i\frac{dw(\theta)/dz}{w(\theta)}=\frac{\overline{b}_{0}rz^2+[b_{1}\overline{r}+i(|r|^2+1)] z+b_{0}\overline{r}}{\overline{r}z(z-r)(z-1/\overline{r})}, \quad z=e^{i\theta},
\end{equation}
with		
\begin{align} \label{def_fr}
b_{1}\overline{r}+i(|r|^2+1)=
\begin{cases} 
\frac{\Re(b_{1})|r|^2-\Im(r)(|r^2|+1)}{\Re(r)}, & \mbox{if } \Re(r)\neq 0,\\[0.5ex]  
\Im(r)\Im(b_{1}), & \mbox{if } \Re(r)= 0.
\end{cases} 
\end{align}				
Observe that $b_{1}\overline{r}+i(|r|^2+1)$ is real. Integrating \eqref{frfrfrfr} with respect to $z=e^{i\theta}$ 
and considering $z_{0}=e^{i\theta_{0}}, \theta_{0} \in [0,2\pi]$,
we get
\begin{equation*} 
i\ln\left(\frac{w(\theta)}{w(\theta_{0})}\right) = \frac{b_{0}\overline{r}}{r}i(\theta-\theta_{0}) + H\ln\left(\frac{z-r}{z_{0}-r}\right)-\overline{H}\ln\left(\frac{z-1/\overline{r}}{z_{0}-1/\overline{r}}\right),
\end{equation*}
with
\begin{equation*}
H=\frac{\overline{b}_0r^2+ b_{1}\overline{r}+i(|r|^2+1)}{|r|^2-1} + \frac{b_{0}(\overline{r})^2}{|r^2|(|r|^2-1)}.
\end{equation*}
Since 
$$ \frac{z-r}{z_{0}-r} =\frac{z}{z_{0}}\frac{1-r\overline{z}}{1-r\overline{z}_{0}}		
\quad \mbox{and} \quad
\frac{z-1/\overline{r}}{z_{0}-1/\overline{r}}
=\frac{z}{z_{0}}\frac{\overline{r}-\overline{z}}{\overline{r}-\overline{z}_{0}}, $$			
it follows that
\begin{equation*}
\ln\left(\frac{w(\theta)}{w(\theta_{0})}\right)
= \left(\frac{b_{0}\overline{r}}{r}-\overline{H}\right)(\theta-\theta_0) + i(\overline{H}-H)\ln\left|\frac{z-r}{z_0-r}\right| + (H+\overline{H})\arg\left(\frac{z-r}{z_{0}-r}\right).
\end{equation*}			
Straightforward manipulations lead to
\begin{equation*}
G=\frac{b_{0}\overline{r}}{r}-\overline{H}=-\frac{2 \Re(b_{0}\overline{r}/r)+b_{1}\overline{r}+i(|r|^2+1)}{|r|^2-1} \in \mathbb{R}.
\end{equation*}		
Hence,		
\begin{equation*}
\Im(H)=\Im\left(\frac{\overline{b}_0{r}}{\overline{r}}\right) 
\quad \mbox{and} \quad 
\Re(H)=\frac{(|r|^2+1)\Re(\overline{b}_0{r}/\overline{r})+b_{1}\overline{r}+i(|r|^2+1)}{|r|^2-1}.
\end{equation*}

Therefore, the weight function can be written as 
\begin{equation*}
w(\theta)  = \tau(G;H) \, e^{[G+2\Re(H)]\theta} \, e^{2\Re(H)\arg(1-\overline{z}r)}\left|z-r\right|^{2\Im(H)},
\end{equation*}
where $\tau(G,H)=w(\theta_{0})e^{-[G+2\Re(H)]\theta_{0}}e^{-2\Re(H)\arg(1-\overline{z}_{0}r)}\left|z_{0}-r\right|^{-2\Im(H)}$.
				
To verify the condition $A(1)[w(2\pi)-w(0)]=0$, it is necessary that $G+2\Re(H)=0$, which is equivalent to
$  b_{1}\overline{r}+i(|r|^2+1) = - 2 \Re (\overline{b}_{0}r^2).$
From \eqref{def_fr} we obtain
\begin{equation*} 
b_{1}= -\frac{2 \Re(b_{0}\overline{r}^2)+i(|r|^2+1)}{\overline{r}}  
\quad \mbox{and} \quad
H =-\frac{{b}_{0}\overline{r}}{r}.
\end{equation*}			
Finally the weight function is given by	
\begin{equation} \label{Medida-tipo-Teorema-3.2-ranga}
w(\theta)=\tau(b_{0},r)e^{-2\Re({b}_{0}\overline{r}/{r})\arg (1-\overline{z}r)}|z-r|^{-2\Im({b}_{0}\overline{r}/{r})}, 
\end{equation}
with $\tau(b_{0},r)$ a constant. Furthermore, this weight function satisfies the Pearson-type equation \eqref{Eq-Tipo-Pearson-1} with $A(z)= (z-r)(z-1/\overline{r})$ 
and $ B(z)= \left(2i + \overline{b}_{0}r/\overline{r}\right)z^2-[2 \Re(b_{0}\overline{r}^2)+(|r|^2+1)i]z/\overline{r}+b_{0}.$
 
\begin{remark} \label{remark39}
Choosing  
$b_0 = - \overline{u}/ \overline{r}$ in \eqref{Medida-tipo-Teorema-3.2-ranga}, we recover the known semi-classical weight function \eqref{Teo3.2}, with the polynomials 
$A(z)= (z-r)(z-1/\overline{r})$ and 
$B(z)=\left(2i- u/\overline{r}\right)z^2+[2\Re(ur)-(|r|^2+1)i]z/\overline{r} -\overline{u}/\overline{r}.$
The Pearson-type equation for the weight function \eqref{Medida-tipo-Teorema-3.2-ranga} with parameter
$b_0 = 2(r/\overline{r})i$, has appeared in \cite{CaSu97} in terms of moment functional, and $B(z) = i [-(|r|^2+1)z + 2r]/\overline{r}$.  The weight function \eqref{Medida-tipo-Teorema-3.2-ranga} belongs to the class $(2,2)$.
\end{remark}

\item[2.] Now we consider  $|r_1| \neq 1$ and  $r_2 \neq 1/\overline{r}_1$. Let $\mathfrak{u}$ and $\mathfrak{v}$ be given as
\begin{align*}
\mathfrak{u} 
&=(1+|r_1^2|)\Im(1/\overline{r}_1)-(1-|r_1|^2)\Re(1/\overline{r}_1)b_{0}, \\[0.5ex]
\mathfrak{v} 
& = (\overline{r}_1+\overline{r}_2)(r_1r_2+1)-(|r_1|^2+1)\Re(r_2+1/\overline{r}_1)-(1-|r_1|^2) \Im[\overline{b}_{0}(r_2-1/\overline{r}_1)].
\end{align*}
Notice that if $r_2=-1/r_1$ and $\Re(r_1)=0$, then $r_2=1/\overline{r}_1$, this possibility was considered before. Therefore, the other solutions of the linear system \eqref{Sistema-A-grau2} are
\begin{align*}
& b_{2}=\overline{b}_{0} r_1r_2 +2i, \\
& b_{1} = 
\begin{cases}  
\frac{(r_1+r_2) \mathfrak{v}}{\Im[(r_1+r_2)(\overline{r}_1\overline{r}_2+1)]},  
& \mbox{if } r_2  \neq -1/{r_1}, \\[0.5ex]
\frac{(\overline{r}_1-|r_1|^2{r_1})\mathfrak{u}+|r_1^2-1|^2 i}{\Re(r_1)(1-|r_1|^2)},  
& \mbox{if } \Re({r}_1) \neq 0  \mbox{ and } r_2=-1/r_1,
\end{cases} \\
& b_{0}= 
\begin{cases}
\frac{\Re(b_{0})(r_2-1/\overline{r}_1)^2}{\Re[(r_2-1/\overline{r}_1)^2]}, & \mbox{if }  \Re[(r_2-{1}/{\overline{r}_1})^2]\neq 0,\\[0.5ex]
\Im(b_{0}) i, & \mbox{if }   \Re[(r_2-{1}/{\overline{r}_1})^2]= 0.
\end{cases}
\end{align*}
Substituting these values in \eqref{Eq-Tipo-Pearson-4} we get
\begin{equation*} 
i \frac{d w(\theta)/d z}{w(\theta)} =
\frac{\overline{b}_0 r_1r_2 z^2 + [b_{1}+(r_1+r_2)i]z + b_0}{z(z-r_1)(z-r_2)}, \quad z=e^{i\theta}.
\end{equation*}
Notice that 
\begin{equation*} 
\overline{b}_{0}r_1r_2=
\begin{cases}
-\frac{\Re(b_{0})|r_2-1/\overline{r}_1|^2}{\Re[(r_2-1/\overline{r}_1)^2]}, & \mbox{if } \Re[(r_2-{1}/{\overline{r}_1})^2]\neq 0,\\[0.5ex]
-(ir_1r_2)\Im(b_{0}), & \mbox{if } \Re[(r_2-{1}/{\overline{r}_1})^2]= 0.
\end{cases}
\end{equation*}
Therefore, $\overline{b}_{0}r_1r_2 \in \mathbb{R}$. We thus can write
\begin{equation*} 
b_{1}+(r_1+r_2)i=
\begin{cases}  
\frac{(r_1+r_2)\{\mathfrak{v}-\Im[(\overline{r}_1+\overline{r}_2)(r_1r_2+1)]i\}}{\Im[(r_1+r_2)(\overline{r}_1\overline{r}_2+1)]},  & \mbox{if }  r_2  \neq -1/{r_1}, \\[0.5ex]
\frac{(r_1^2-1)b_{0}}{r_1},  & \mbox{if } r_2=-1/r_1 \mbox{  and } \Re(r_1)\neq 0.
\end{cases}
\end{equation*}
Using integration once more, with $z_{0}=e^{i\theta_{0}}$, we get
\begin{equation*} 
\ln\left(\frac{w(\theta)}{w(\theta_{0})}\right) = b_{0}\overline{r}_1\overline{r}_2(\theta-\theta_{0}) +iH
\left[\ln\left(\frac{z-r_2}{z_{0}-r_2}\right) -
\ln\left(\frac{z-r_1}{z_{0}-r_1}\right)  \right],
\end{equation*}
where
$H=[\overline{b}_{0}r_1 r_2  (r_1+r_2)+b_{1}+(r_1+r_2)i]/(r_1-r_2)=0.$
Hence, the weight function is
$w(\theta)=w(\theta_{0})e^{b_{0}\overline{r}_1\overline{r}_2(\theta-\theta_{0})}.$
From the condition  $A(1)[w(2\pi)-w(0)]=0,$ it follows that ${b}_{0}=0$.
Hence, we obtain the Lebesgue weight function.
\end{description}

\subsubsection{$A(z)= (z-r_1)(z-r_2)$ with $|r_1| \neq 1$, $|r_2| \neq 1$  and $|r_1r_2| \neq 1$}

In this case the solutions of the linear system \eqref{Sistema-A-grau2} are 
\begin{equation*} 
b_{2}=\Re(b_{2})+2i, \quad
b_{1}= -(r_1+r_2)[\Re(b_{2})+i] \quad \mbox{and} \quad
b_{0}=r_1r_2\Re(b_{2}),
\end{equation*}
with $\Re(b_{2})$  an arbitrary real number.

This solution  includes also the cases $r_1=r_2=0$ or $r_1=0$ and $r_2 \neq 0$ with  $|r_2| \neq 1$, i.e., when $A(z)=z^2$ or $A(z)=z(z-r)$, with $|r| \neq 1$.
The representation \eqref{Eq-Tipo-Pearson-4} becomes
$ (idw(\theta)/dz)/w(\theta)=\Re(b_{2})/z.$
Again, assuming the condition $A(1)[w(2\pi)-w(0)]=0$, we obtain the Lebesgue weight function.

\setcounter{equation}{0}
\section{Applications} 
\label{sec_Application}

In this section we provide structure relations for the MOPUC (using Theorem \ref{prop_rel_est}) and non-linear difference equations for Verblunsky coefficients using Theorem \ref{TEO_eq_de_dif_caso_geral} associated with examples of semi-classical weight function on the unit circle.

\subsection{Example 1} 
\label{subsection_exemplo1}

Here we consider the semi-classical weight function  \eqref{Teo3.3}, studied in \cite{Ra10}, namely
\begin{equation} \label{wf_interesse_teta}
w(\theta)=\tau(b)e^{-\eta \theta}[\sin^2(\theta/2)]^{\lambda},
\end{equation}
where  $\eta \in \mathbb{R}$, $\lambda > -1/2$, $b=\lambda+i\eta$, and  $  \tau(b)= e^{\pi\eta} 2^{b+\overline{b}}|\Gamma(b+1)|^2/[2\pi\Gamma(b+\overline{b}+1)]. $

The monic orthogonal polynomials, $\Phi_n(z)$, can be expressed in terms of hypergeometric functions as
\begin{equation*} 
\Phi_n(z)  =\frac{(b+\overline{b}+1)_n}{(b+1)_n} {}_2F_1(-n,b+1,b+\overline{b}+1;1-z),  \quad n \geqslant 0.
\end{equation*}
see \cite{Ra10}. 
The Verblunsky coefficients are given as
\begin{equation} \label{alfas}
\alpha_{n} = - \frac{(b)_{n+1}}{(\overline{b}+1)_{n+1}}, \quad n \geqslant 0.
\end{equation}

In Section \ref{sec_Characterization} we determined the polynomials $A$ and $B$, with $A(1)=0$, of the Pearson-type equation \eqref{Eq-Tipo-Pearson-1}, see Remarks \ref{remark33}, \ref{remark34}, \ref{remark35}, \ref{remark37}, and \ref{remark38}. The possibilities are
\begin{enumerate}
\item When $A(z)$ has degree 1, $A(z)=z-1$ and
$B(z) = i [(b+1)z+\overline{b}].$ 
Hence, the weight function \eqref{wf_interesse_teta} belongs to the class $(1,1)$.

\item When $A(z)$ has degree 2, $A(z)=(z-1)(z-r)$, $r \in \mathbb{C}$ and
$ B(z)  = i[(b+2)z^2 + [\overline{b}-1   -r(b+1)]z  - r\overline{b}].$
For special values of $r$:

i) If $r=-1$, then  $A(z)=z^2-1$ and
$ B(z) = i [(b+2)z^{2} +(b+\overline{b})z +\overline{b}].$

ii) If $r=1$, then $A(z)=(z-1)^2$ and
$B(z) = i [(b+2)z^{2} +(\overline{b}-b-2)z -\overline{b}].$

iii) If $r=0$, then  $A(z)=z(z-1)$ and 
$B(z) = i [(b+2)z^{2} +(\overline{b}-1)z].$

Here we observe that the weight function \eqref{wf_interesse_teta} also  belongs to the class $(2,2)$, as expected.
\end{enumerate}

We remark that this semi-classical weight function satisfies several Pearson-type equations. In the next subsections we provide the structure relations for the monic orthogonal polynomials $\Phi_n$  and the non-linear difference equations for the Verblunsky coefficients, $\alpha_n$, corresponding to the different choices of polynomial $A$. 

\subsubsection{Structure relations}

Using Theorem \ref{prop_rel_est} and Corollary \ref{coro2}, with the different choices for the polynomials $A$ and $B$, we can obtain several structure relations for the MOPUC.

With $A(z)= z-1$, the equations \eqref{antiga50} and \eqref{antiga49} for $\s_{n,n}$  and \eqref{lnn1} yield
\begin{equation} \label{l_b_barra}
\gamma_{n}  = \overline{b} - (b+n-1)\overline{\alpha}_{n-1} \alpha_{n-2}.
\end{equation}

From Theorem \ref{prop_rel_est} and using the relation \eqref{l_b_barra},  we present several structure relations for the MOPUC associated with the weight function \eqref{wf_interesse_teta} in the  next result.

\begin{theorem} \label{TheoSR}
The MOPUC associated with the semi-classical weight function \eqref{wf_interesse_teta} satisfy the following structure relations, for $n \geqslant 2$.

\noindent {\rm i)} For $A(z) = z-1$,
\begin{equation} \label{SR1}
(z-1)\Phi_{n}'(z)= -(\overline{b}+n)[1-|\alpha_{n-1}|^2]\Phi_{n-1}(z)+ n \Phi_{n}(z).
\end{equation}

\noindent {\rm ii)} For $A(z) = (z-1)(z-r)$, $r \in \mathbb{C}$,
\begin{equation*} 
(z-1)(z-r)\Phi_{n}'(z)=
n\Phi_{n+1}(z)
-[\overline{b} +n(r+1)] \Phi_{n}(z)
+r(\overline{b}+n)[1-|\alpha_{n-1}|^2] \Phi_{n-1}(z)
-(b+1)\overline{\alpha}_{n}\Phi_{n}^{\ast}(z),
\end{equation*}

\noindent {$\bullet$} For $A(z) = z(z-1)$,
\begin{equation} \label{SR4}
z(z-1)\Phi_{n}'(z)=   
 n\Phi_{n+1}(z) 
 -(\overline{b}+n) \Phi_{n}(z)
 -(b+1)\overline{\alpha}_{n}\Phi_{n}^{\ast}(z).
\end{equation}

\noindent {$\bullet$} For $A(z) = (z-1)^2$,
\begin{equation} \label{SR3}
(z-1)^2\Phi_{n}'(z)= 
 n\Phi_{n+1}(z)
 -(\overline{b}+2n) \Phi_{n}(z)
 +(\overline{b}+n)[1-|\alpha_{n-1}|^2]\Phi_{n-1}(z)
 -(b+1) \overline{\alpha}_{n}\Phi_{n}^{\ast}(z),
\end{equation}

\noindent {$\bullet$} For $A(z) = z^2-1$,
\begin{equation} \label{SR2}
(z^2-1)\Phi_{n}'(z)=
n \Phi_{n+1}(z)
-\overline{b} \, \Phi_{n}(z)
-(\overline{b}+n)[1-|\alpha_{n-1}|^2]\Phi_{n-1}(z) 
-(b+1)\overline{\alpha}_{n}\Phi_{n}^{\ast}(z).
\end{equation}
\end{theorem}

We now observe the explicit value of $\gamma_n$ in \eqref{l_b_barra}, for $n+1$ and using \eqref{lnn1} we have
$\overline{b}  =  \gamma_{n} + (b+n+1)\overline{\alpha}_{n}\alpha_{n-1}.$
On the other hand, from \eqref{l_b_barra}, we get
$\overline{b}  =   \gamma_{n} +  (b+n-1) \overline{\alpha}_{n-1}\alpha_{n-2}.$
Hence,
\begin{equation}  \label{1111}
\alpha_{n}\overline{\alpha}_{n-1} =
\frac{(\overline{b}+n-1)}{(\overline{b}+n+1)} {\alpha_{n-1}\overline{\alpha}_{n-2}}.
\end{equation}
Therefore, we can show the following results.

\begin{proposition}
The coefficients $\alpha_n$ and $ \gamma_{n}$ satisfy
\begin{equation} \label{lnn1-novo}
\alpha_{n}\overline{\alpha}_{n-1} =    \frac{ |b|^2  }{(\overline{b}+n)(\overline{b}+n+1)}
\quad \mbox{and} \quad 
 \gamma_{n} = \frac{n\overline{b}}{b+n}. 
\end{equation}
\end{proposition}
\begin{proof}
Since $ \alpha_{0} = -b/(\overline{b}+1)$ and $\alpha_{-1} = -1$, using \eqref{1111} we get
the first relation  in \eqref{lnn1-novo}. 

Second equation in \eqref {lnn1-novo} follows since  $\gamma_{n} + (b+n+1)\overline{\alpha}_{n} \alpha_{n-1}  = \overline{b}.$ 
\end{proof}

\subsubsection{Non-linear difference equations for complex Verblunsky coefficients}

Using Theorem \ref{TEO_eq_de_dif_caso_geral} 
we find non-linear difference equations for complex Verblunsky coefficients associated with the weight function \eqref{wf_interesse_teta}.

\begin{theorem}  \label{Theoalphas_1}
The Verblunsky coefficients satisfy the following non-linear difference equations
\begin{equation} \label{ref_dif_A}
(b+n+1) \alpha_{n} = [n+2+ \overline{b} - (b+n+1)\overline{\alpha}_{n+1} \alpha_{n}] \dfrac{\alpha_{n+1}}{1-|\alpha_{n+1}|^2}, \quad   n \geqslant 0,
\end{equation}
\begin{equation} \label{caso_4b}
(\overline{b}+n+1)\alpha_n = (b+n)  \alpha_{n-1},
\quad  n \geqslant 1,
\end{equation}
\begin{equation} \label{caso_3a}
(\overline{b}+n+1) \alpha_n + (b+n-1) \alpha_{n-2} =
 \frac{2[\Re(\gamma_{n}) +n]\alpha_{n-1} \, }{1- |\alpha_{n-1}|^2}, \quad  n \geqslant 2,
\end{equation}
\begin{equation} \label{caso_2a}
(\overline{b}+n+1)\alpha_{n} - (b+n-1)\alpha_{n-2} = -  \dfrac{2i \Im(\gamma_n) \, \alpha_{n-1}}{1-|\alpha_{n-1}|^2}, \quad   n \geqslant 2,
\end{equation}
\begin{equation} \label{RD_caso3}
( \overline{b}+n+1)\alpha_n  
+(b+n-1)(1-|\alpha_{n-1}|^2) \alpha_{n-2}
= \left(\frac{n\overline{b}}{b+n} +  b+2n \right) \alpha_{n-1}, \quad n \geqslant 2,
\end{equation}
and
\begin{equation} \label{RD_caso2}
( \overline{b}+n+1)\alpha_n - (b+n-1)(1-|\alpha_{n-1}|^2) \alpha_{n-2} = - \left(\frac{n\overline{b}}{b+n} - b\right) \, \alpha_{n-1}, 
\quad n \geqslant 2.
\end{equation}
\end{theorem}

\begin{proof} $ $ 

$\bullet$ When $A(z)=z-1$ and
$B(z) = i [(b+1)z+\overline{b}],$ equation \eqref{eq_de_dif_caso_geral} yields 
\begin{equation} \label{caso_1a}
(b+n-1)\alpha_{n-2} = (n+\gamma_{n}) \dfrac{\alpha_{n-1}}{1-|\alpha_{n-1}|^2}, \quad n \geqslant 2,
\end{equation}
and using \eqref{l_b_barra}  we obtain \eqref{ref_dif_A}. 

$\bullet$ For $A(z)=z(z-1)$ and
$ B(z) = i [(b+2)z^{2} +(\overline{b}-1)z],$  equation \eqref{eq_de_dif_caso_geral} is 
\begin{equation} \label{caso_4a}
(\overline{b}+n+1)\alpha_n =
(n+\overline{\gamma}_{n}) \dfrac{\alpha_{n-1}}{1-|\alpha_{n-1}|^2} \quad n \geqslant 1,
\end{equation}
and relation \eqref{eq_de_dif_caso_geral_2} gives  \eqref{caso_4b}.

$\bullet$ If $A(z)=(z-1)^2$ and
$B(z) = i [(b+2)z^{2} +(\overline{b}-b-2)z -\overline{b}],$ \eqref{eq_de_dif_caso_geral} yields \eqref{caso_3a}.

The relation \eqref{eq_de_dif_caso_geral_2} becomes
\begin{equation} \label{caso_3b}
(\overline{b}+n+1)\alpha_n +
(b+n-1)(1-|\alpha_{n-1}|^2) \alpha_{n-2}
= \left(\gamma_n +  b+2n \right) \, \alpha_{n-1}. 
\end{equation}
Using the second equation in \eqref{lnn1-novo}, the  equation \eqref{caso_3b} can be written as \eqref{RD_caso3}.

$\bullet$ When $A(z)=z^2-1$ and
$B(z) = i [(b+2)z^{2} +(b+\overline{b})z +\overline{b}],$ 
 \eqref{eq_de_dif_caso_geral} yields \eqref{caso_2a}.

From relation \eqref{eq_de_dif_caso_geral_2} we get
\begin{equation*}
(\overline{b}+n+1)\alpha_{n} -  (b+n-1)\alpha_{n-2}  (1-|\alpha_{n-1}|^2) =
- [2\gamma_n -b-\overline{b} +(b+n+1)\overline{\alpha}_{n}\alpha_{n-1}] \alpha_{n-1}.
\end{equation*}
Using \eqref{lnn1} and \eqref{l_b_barra}, the above equation implies to
\begin{equation} \label{caso_2b}
(\overline{b}+n+1)\alpha_{n} - (b+n-1)\alpha_{n-2}  (1-|\alpha_{n-1}|^2) = - ( \gamma_n -b) \alpha_{n-1},
\end{equation}
we thus obtain \eqref{RD_caso2} from \eqref{lnn1-novo}.

We remark that 
it is possible  to obtain  \eqref{caso_4b} by adding  \eqref{caso_3b} and \eqref{caso_2b}.
It is possible to obtain \eqref{caso_2a} by subtracting \eqref{caso_4a} from \eqref{caso_1a}.  
\end{proof}

\begin{remark}
Using the non-linear difference equation \eqref{caso_4b}, we can write
\begin{equation*}
\alpha_n = \dfrac{b+n}{\overline{b}+n+1}\dfrac{b+n-1}{\overline{b}+n}\cdots\dfrac{b}{\overline{b}+1}\alpha_{-1}, \quad n \geqslant 0,
\end{equation*}
and since $\alpha_{-1}=-1$, we obtain the explicit formula for the Verblunsky coefficients \eqref{alfas}.
\end{remark}

Furthermore, using another method (see \cite{Va08}), we find one more non-linear difference equation for the complex Verblunsky coefficients associated with the weight function \eqref{wf_interesse_teta}.

\begin{theorem}  
The Verblunsky coefficients satisfy the following non-linear difference equation
\begin{equation} \label{RD_caso1}
\overline{\alpha}_{n} ( \alpha_{n-1}  + \alpha_{n} ) =
 (\overline{b}+n) (|\alpha_{n-1}|^2 -|\alpha_{n}|^2 ),
 \quad n \geqslant 1. 
\end{equation}
\end{theorem}
\begin{proof}
  From relation \eqref{Szegorecorrence} we get
\begin{equation*} 
\Phi_{n+1}^\prime(z) =  z\Phi_{n}^\prime(z) + \Phi_{n}(z) - \overline{\alpha}_{n} \Phi_{n}^{* \prime}(z)
\end{equation*}
using the structure relation \eqref{SR1}, 
\begin{equation*}
 \frac{(n+1) \Phi_{n+1}(z) + \s_{n+1,n} \Phi_{n}(z)}{z-1} =
 \frac{z\left[n \Phi_n(z) +\s_{n,n-1}  \Phi_{n-1}(z)\right]}{z-1}  + \frac{(z-1)\left[ \Phi_{n}(z) - \overline{\alpha}_{n} \Phi_{n}^{* \prime}(z) \right]}{z-1},
\end{equation*}
where $\s_{n+1,n} = -(\overline{b}+n+1)[1-|\alpha_{n}|^2]$ and  $\s_{n,n-1} = -(\overline{b}+n)[1-|\alpha_{n-1}|^2]$.

Again using relation \eqref{Szegorecorrence}, we get
\begin{align*}
(n+1) \Phi_{n+1}(z) + \s_{n+1,n} \Phi_{n}(z) =
& (n+1)[\Phi_{n+1}(z) + \overline{\alpha}_{n} \Phi_{n}^{*}(z)] + \s_{n,n-1} [\Phi_{n}(z) + \overline{\alpha}_{n-1} \Phi_{n-1}^{*}(z)] \\[0.5ex]
& - \overline{\alpha}_{n} z \Phi_{n}^{* \prime}(z)  - \Phi_{n}(z) + \overline{\alpha}_{n} \Phi_{n}^{* \prime}(z).
\end{align*}
Since  $\Phi_{n}^{*}(z) = - \alpha_{n-1} z^n + \cdots $  and 
$z \Phi_{n}^{* \prime}(z) = - \alpha_{n-1} n z^n + \cdots,$
it follows that
\begin{align*}
 \s_{n+1,n} \Phi_{n}(z) = 
& (n+1)  \overline{\alpha}_{n} \left[  - \alpha_{n-1} z^n + \cdots \right]  + \s_{n,n-1}\left[\Phi_{n}(z) + \overline{\alpha}_{n-1} \Phi_{n-1}^{*}(z) \right] \\[0.5ex]
&  +\left[  \overline{\alpha}_{n} \alpha_{n-1} n z^n + \cdots \right] - \Phi_{n}(z) + \overline{\alpha}_{n} \Phi_{n}^{* \prime}(z).
\end{align*}
Comparing the coefficients of $z^{n}$, we obtain
$ \s_{n+1,n} = - (n+1)  \overline{\alpha}_{n}  \alpha_{n-1} + \s_{n,n-1} +  \overline{\alpha}_{n} \alpha_{n-1} n -1$
and
$ \s_{n,n-1} - \s_{n+1,n} = \overline{\alpha}_{n} \alpha_{n-1}  + 1. $
Now, one can write
\begin{equation*}
 \overline{\alpha}_{n} \alpha_{n-1} 
= (\overline{b}+n+1)(1-|\alpha_{n}|^2) - 
(\overline{b}+n) (1- |\alpha_{n-1}|^2) - 1
\end{equation*}
and, finally \eqref{RD_caso1} is obtained.
\end{proof} 

\subsubsection{Special case}
Here we consider the semi-classical weight function  \eqref{exp_cos}, namely, for $\eta \in \mathbb{R}$ and $\beta > -1/2$,
\begin{equation*} 
w(\theta)=\tau(\beta,\eta) \, e^{-\theta \, \eta}\, [\cos^2(\theta/2)]^{\beta}, \quad \theta \in [-\pi,\pi].
\end{equation*}

This weight function satisfies the Pearson-type equation \eqref{Eq-Tipo-Pearson-1} with several polynomials $A$ and $B$. Denoting $c = \beta + i \eta$, we can calculate these polynomials:

\begin{enumerate}
\item[i)] When $A(z)$ has degree 1, $A(z)=z+1$ and
$B(z) = i [(c+1)z-\overline{c}].$

\item[ii)] When $A(z)$ has degree 2, $A(z)=(z+1)(z-r)$ and
$ B(z)  = i[(c+2)z^2 - (\overline{c}-1+r(c+1))z + r\overline{c}].$
\end{enumerate}

Similarly, using theorems \ref{prop_rel_est} and \ref{TEO_eq_de_dif_caso_geral}, it is possible to generate several structure relations for the MOPUC and non-linear difference equations for the Verblunsky coefficients. Also we  find that the associated Verblunsky coefficients    
and the MOPUC are
$$\alpha_{n} = (-1)^n \frac{(c)_{n+1}}{(\overline{c}+1)_{n+1}} \quad \mbox{and} \quad \Phi_{n}(z)  = (-1)^n\frac{(c+\overline{c}+1)_n}{(c+1)_n} {}_2F_1(-n,c+1,c+\overline{c}+1;1+z).$$

\subsection{Example 2}
\label{subsection42}

Here we consider the semi-classical weight function  \eqref{Teo3.1}   
\begin{equation*} 
w(\theta) = \tau(u) e^{2|u| \sin(\theta + \arg  u)},
\end{equation*}
where $u \in \mathbb{C}$  and the constant $\tau(u)$ is such that $ \mu_0 =1. $ 
We observe that if $u=0$, then
$ w(\theta) = 1/(2\pi)$, the Lebesgue measure.

Considering $u \neq 0$, as presented in Subsection \ref{subsection321}, see Remark \ref{remark31}, this weight function satisfies Pearson-type equation \eqref{Eq-Tipo-Pearson-1} with $A(z)=z$ and $B(z) = u z^2 + i z + \overline{u}.$ 

From Theorem \ref{prop_rel_est}, the structure relation for the MOPUC is
\begin{equation*} 
z\Phi_{n}'(z)= i \overline{u} (1-|\alpha_{n-1}|^2) \Phi_{n-1}(z) + (n + i u \overline{\alpha}_n \alpha_{n-1})\Phi_{n}(z) + i u \overline{\alpha}_{n} \Phi_{n}^{\ast}(z).
\end{equation*}

Using the definition \eqref{antiga49} for $\s_{n,n}$, we also find that 
$\s_{n,n} = - i \overline{u} \alpha_n \overline{\alpha}_{n-1} +n.$
Therefore, the Verblunsky coefficients satisfy
\begin{equation*}
u \overline{\alpha}_{n} \alpha_{n-1} = - \overline{u} \alpha_n \overline{\alpha}_{n-1}, \quad n \geqslant 2.
\end{equation*}
It means that $\Re (u \overline{\alpha}_{n} \alpha_{n-1}) =0$. Since $u \neq 0$, we find 
\begin{equation*}
\frac{\alpha_n}{\overline{\alpha}_{n}} =  
\frac{\alpha_{1}}{\overline{\alpha}_{1}}
\left( -\frac{u}{\overline{u}} \right)^{n-1}, \quad n \geqslant 2.
\end{equation*}

Moreover, from  
\eqref{eq_de_dif_caso_geral} the Verblunsky coefficients  also satisfy
\begin{equation}  \label{ED_iu}
\overline{i u} \alpha_{n} + i u \alpha_{n-2} =  \frac{n \,\alpha_{n-1}}{1-|\alpha_{n-1}|^2}, \quad n \geq 2.
\end{equation}
Comparing with \eqref{Painleve}, the non-linear difference equation \eqref{ED_iu} can be seen as a complex extension of the Painlev\'{e} discrete equation 
${\rm dP_{II}}$.

Finally, the non-linear difference equation
\eqref{eq_de_dif_caso_geral_2}  for the Verblunsky coefficients  becomes
\begin{equation*} 
 \overline{iu} \, \alpha_{n} + iu \, \alpha_{n-2}({1-|\alpha_{n-1}|^2}) = 
  (n + i u \overline{\alpha}_{n}\alpha_{n-1}) \alpha_{n-1}, \quad n \geq 2.
\end{equation*}

\subsection{Example 3}
 
Here we consider the semi-classical weight function  \eqref{Teo3.2}, namely, for $u, r \in \mathbb{C}$, $|r| \neq 1$, $r \neq 0$ and the constant $\tau(u,r)$ is such that  $ \mu_0=1, $ 
\begin{equation*}
w(\theta)  =  \tau(u, r) \, e^{2\Re(u/\overline{r}) \arg(1 -r e^{-i\theta})} |e^{i\theta}-r|^{-2\Im(u/\overline{r})}.
\end{equation*}

From the first result of Subsection \ref{subsection334} and Remark \ref{remark39}, we know that this weight function satisfies the Pearson-type equation with
$A(z)= (z-r)(z-1/\overline{r}) $ 
and 
$B(z)=\left(2i- u/\overline{r}\right)z^2
+[2\Re(ur)-(|r|^2+1)i]z/\overline{r} 
-\overline{u}/\overline{r}.$

Hence, from \eqref{Est-geral2-nova} the MOPUC satisfy the structure relation
\begin{equation*}
(z-r)\left( z- 1/\overline{r} \right) \Phi_{n}'(z)=
n   \Phi_{n+1}(z)
+ \s_{n,n}\Phi_{n}(z)
+ \s_{n,n-1}\Phi_{n-1}(z)
- \left(1+\frac{u}{\overline{r}} i\right)\overline{\alpha}_n \Phi_{n}^{\ast}(z),
\end{equation*}
with
\begin{equation*}
\s_{n,n-1}= \frac{n r - \overline{u}i}{\overline{r}} (1-|\alpha_{n-1}|^2) \ \mbox{and} \
\s_{n,n}  = - \left[n  \left( \frac{|r|^2 +1}{\overline{r}} \right) + \gamma_n +  \left( n+1 + \frac{u}{\overline{r}}i \right)\overline{\alpha}_n \alpha_{n-1} \right].  
\end{equation*}
Also, from \eqref{antiga49}
\begin{equation*}
 \s_{n,n}   = - \left[ n  \left( \frac{|r|^2 +1}{\overline{r}} \right) 
+ \frac{r}{\overline{r}} \overline{\gamma}_n 
+\left( (n+1)\frac{r}{\overline{r}} - \frac{\overline{u}}{\overline{r}}i \right)\alpha_n \overline{\alpha}_{n-1}
\right] + \frac{2\Re(ur)}{\overline{r}} i.
\end{equation*}
Comparing the two forms to represent the coefficient $\s_{n,n}$, we conclude that   
\begin{equation*}
\Im(\overline{r} \gamma_{n} +  [(n+1)\overline{r}+ i u]\overline{\alpha}_{n} \alpha_{n-1}) =   -\Re(ur).
\end{equation*}

From Theorem \ref{TEO_eq_de_dif_caso_geral}, the difference equation \eqref{eq_de_dif_caso_geral} for the Verblunsky coefficients become
\begin{equation*} 
[(n+1) r -  \overline{u}i]\alpha_n +[(n-1) \overline{r} + ui]  \alpha_{n-2} = \left[n(|r|^2 + 1) + 2\Re( r \overline{\gamma}_{n})\right] \frac{\alpha_{n-1}}{1-|\alpha_{n-1}|^2}, \quad n \geqslant 2.
\end{equation*}
The difference equation
\eqref{eq_de_dif_caso_geral_2}  for the Verblunsky coefficients  is
\begin{align*}
& [(n+1) r - \overline{u}i]\alpha_n + [(n-1)\overline{r}+ui] \alpha_{n-2} (1-|\alpha_{n-1}|^2) \\[0.5ex]
& \qquad \qquad \qquad = \{n(|r|^2 + 1) +2 \Re(ur) i 
+ [(n+1)\overline{r} + ui] \overline{\alpha_n }\alpha_{n-1}
+ 2 \overline{r}\gamma_n\}
\alpha_{n-1}, \quad n \geqslant 2.
\end{align*} 
 
\subsection{Example 4}

Here we consider the semi-classical weight function  \eqref{newfunction},
\begin{equation*} 
w(\theta)=\tau(\lambda,\beta,\eta) \, e^{-\theta \, \eta}\, [\sin^2(\theta/2)]^{\lambda}[\cos^2(\theta/2)]^{\beta}, \quad \theta \in [0,2\pi],
\end{equation*}
where $\tau(\lambda,\beta,\eta)$ is a constant, $\eta \in \mathbb{R}$,  $\lambda >-1/2$ and $\beta > -1/2$, that satisfies the Pearson-type equation \eqref{Eq-Tipo-Pearson-2} with $A(z) =  z^2 -1$ and 
$B(z) = i[(\lambda+\beta + i \eta +2) \, z^2 + 2(\lambda - \beta)z +(\lambda+\beta - i \eta)],$
see Remark \ref{remark35}.

Using Theorem \ref{prop_rel_est} we get the following structure relation for the MOPUC
\begin{equation*} 
(z^2-1)\Phi_{n}'(z) = \s_{n,n-1}\Phi_{n-1}(z)
+\s_{n,n}\Phi_{n}(z) +n\Phi_{n+1}(z)
-(d+1)\overline{\alpha}_{n}\Phi_{n}^{\ast}(z),
\end{equation*}
where  $d= \lambda+\beta + i \eta$, 
\begin{equation*} 
\s_{n,n-1} =  -(\overline{d}+n)(1-|\alpha_{n-1}|^2),
\quad \mbox{ and } \quad
\s_{n,n} = -[(d+n+1)\overline{\alpha}_{n} \alpha_{n-1} + \gamma_{n}].
\end{equation*}
From \eqref{antiga49}, the coefficient $\s_{n,n}$ also can be written as 
\begin{equation*}
\s_{n,n}  = \overline{\gamma}_{n} +(\overline{d}+n+1)\alpha_{n} \overline{\alpha}_{n-1} -2(\lambda - \beta).
\end{equation*}

Comparing the two forms to represent the coefficient $\s_{n,n}$, one sees that
$\Re(\gamma_n + (d+n+1)\overline{\alpha}_{n} \alpha_{n-1}) =  \lambda - \beta.$

Using Theorem \ref{TEO_eq_de_dif_caso_geral} we can provide a non-linear difference equations for the complex Verblunsky coefficients. From \eqref{eq_de_dif_caso_geral} and \eqref{eq_de_dif_caso_geral_2} we obtain, respectively, for   $n \geqslant 2,$
\begin{equation*}
(\overline{d}+n+1) \alpha_n
- (d+n-1) \alpha_{n-2} 
= -  \dfrac{2i \Im(\gamma_n) \,  \alpha_{n-1}}{1-|\alpha_{n-1}|^2}, 
\end{equation*}
and
$(\overline{d}+n+1) \alpha_n 
- (d+n-1) \alpha_{n-2} ({1-|\alpha_{n-1}|^2}) =
- \{ 2\gamma_{n}   +(d+n+1) \overline{\alpha}_{n} \alpha_{n-1}  - 2(\lambda-\beta) \}  \alpha_{n-1}.$

\subsection{Example 5: Real Verblunsky coefficients}

\subsubsection{Jacobi polynomials on the unit circle} 

Choosing $\eta=0$ in \eqref{newfunction} or in Example 4, we have
\begin{equation*}
w(\theta) = \tau(\lambda,\beta)  [\sin^2(\theta/2)]^{\lambda} [\cos^2(\theta/2)]^{\beta}, \quad \lambda, \beta > -1/2, \quad \theta \in [0,2\pi].
\end{equation*}
This weight function was studied  in \cite{Ma00} considering the Jacobi polynomials on the unit circle.
The Verblunsky coefficients are real and given by 
\begin{equation*}
\alpha_{n} = - \Phi_{n+1}(0) = - \frac{\lambda + (-1)^{n+1} \beta}{n+1+\lambda+\beta}.
\end{equation*}
 Observe that the coefficients $\gamma_n$ given by \eqref{lnn1} are also real. 

Since $\eta=0$, the polynomial $B(z)$ becomes 
$B(z) = i[(\lambda+\beta  +2) \, z^2 + 2(\lambda - \beta)z +(\lambda+\beta)]. $
The structure relation for the MOPUC is
\begin{equation*} 
(z^2-1)\Phi_{n}'(z) = \s_{n,n-1}\Phi_{n-1}(z)
+\s_{n,n}\Phi_{n}(z) +n\Phi_{n+1}(z)
-(\lambda+\beta+1) \alpha_n\Phi_{n}^{\ast}(z),
\end{equation*}
where
\begin{equation*}
\s_{n,n-1}   =  -(\lambda+\beta+n)(1-\alpha_{n-1}^2) 
\quad \mbox{and} \quad
\s_{n,n} = -[(\lambda+\beta+n+1)\alpha_n \alpha_{n-1} + \gamma_{n}],
\end{equation*}
and also
$\s_{n,n}  
=\gamma_{n} +(\lambda+\beta+n+1)\alpha_{n} \alpha_{n-1} -2(\lambda - \beta).$

Finally, using Theorem \ref{Theoalphas_1}, we obtain non-linear difference equations for the  Verblunsky coefficients:
\begin{equation*}
(\lambda+\beta+n+1) \alpha_n 
- (\lambda+\beta+n-1) \alpha_{n-2} ({1-\alpha_{n-1}^2}) = - [2\gamma_{n}   +(\lambda+\beta+n+1) \alpha_{n} \alpha_{n-1}  - 2(\lambda-\beta)]  \alpha_{n-1},
\end{equation*}
and
\begin{equation*} 
(\lambda+\beta +n+1) \alpha_n
= (\lambda+\beta +n-1) \alpha_{n-2}, \quad n \geqslant 2.
\end{equation*}
A similar difference equation to the latter one  has appeared in \cite{Ma00}.

\subsubsection{Circular Jacobi polynomials} 

Choosing $\eta=0$ and $\beta=0$ in \eqref{newfunction} or in Example 4, we get
$ w(\theta)=  \tau(\lambda) 
[\sin^2(\theta/2)]^{\lambda}, $ $ \lambda > -1/2,$  $\theta \in [0,2\pi].$
From \eqref{caso_4b}, the Verblunsky coefficients satisfy  
\begin{equation*}
\alpha_{n} = \frac{n+\lambda}{n+1+\lambda} \alpha_{n-1}
\quad \mbox{and} \quad
\alpha_{n} =  -  \frac{\lambda}{{n+1+\lambda}}.
\end{equation*}
Notice that from Theorem \ref{Theoalphas_1} one obtain other different non-linear difference equations for the Verblunsky coefficients.

It is well known that the monic circular Jacobi polynomials satisfy the structure relation
\begin{equation} \label{magnusSR}
(z-1) \Phi_{n}^{\prime}(z) = n \Phi_{n}(z) - \frac{n(n+2\lambda)}{n+\lambda}  \Phi_{n-1}(z),
\end{equation}
 see  \cite{Is05}, \cite{IsWi01}, and \cite{Ma00}. Using Theorem \ref{TheoSR} we also obtain \eqref{magnusSR}. Furthermore, from \eqref{SR4}, \eqref{SR3}, and \eqref{SR2} it follows, respectively, other structure relations for the MOPUC
\begin{equation*} 
z(z-1)\Phi_{n}'(z)=   -(\lambda+n) \Phi_{n}(z)+n\Phi_{n+1}(z) 
+\frac{\lambda(\lambda+1)}{n+1+\lambda}\Phi_{n}^{\ast}(z),
\end{equation*}
\begin{equation*}
(z-1)^2\Phi_{n}'(z)=
\frac{n(n+2\lambda)}{n+\lambda}\Phi_{n-1}(z) -(\lambda+2n) \Phi_{n}(z)+ n \Phi_{n+1}(z)
+\frac{\lambda(\lambda+1)}{n+1+\lambda}\Phi_{n}^{\ast}(z)
\end{equation*}
and
\begin{equation*}
(z^2-1)\Phi_{n}'(z)=
- \frac{n(n+2\lambda)}{n+\lambda}\Phi_{n-1}(z) -\lambda \Phi_{n}(z)+ n \Phi_{n+1}(z)
+\frac{\lambda(\lambda+1)}{n+1+\lambda}\Phi_{n}^{\ast}(z).
\end{equation*}

\subsubsection{Modified Bessel polynomials} 

If $u= i t/2$ with $t>0$, see Remark \ref{remark32}, then $\sin(\theta + \arg  u) = \sin(\theta + \pi/2) = \cos(\theta) $,
\begin{equation*}
w(\theta) = \frac{1}{2\pi} e^{t \cos(\theta)} \quad \mbox{ and }  \quad
\nu(z) = \frac{1}{2\pi} e^{t (z+z^{-1})/2}.
\end{equation*}
In this case the Verblunsky coefficients are real.
The associated weight function $\nu$ satisfies  
$\dfrac{d \nu(z)}{dz} = \frac{t}{2} \left( 1 - \frac{1}{z^2} \right) \nu(z),$
and the MOPUC satisfy the structure relation \eqref{rel_est_conhecida}, mentioned in the introduction, see \cite{Va08}.

From Remark \ref{remark32}, the weight function $w$ satisfies the Pearson-type equation \eqref{Eq-Tipo-Pearson-2}  with  $A(z)=z$ and $B(z) = i \left[ (t/2)z^2 + z -t/2 \right]$. From Corollary \ref{coro2} and $ \kappa_{n-1}^2/\kappa_{n}^2 = 1-|\alpha_{n-1}|^2$, we obtain the following structure relation 
\begin{equation}   \label{SR_real}
 z\Phi_{n}'(z)= n\Phi_{n}(z) + \frac{t}{2}
\frac{\kappa_{n-1}^2}{\kappa_{n}^2}
\left[ \Phi_{n-1}(z)-\alpha_n \Phi_{n-1}^{*}(z) \right].
\end{equation} 
Making $z=0$ in \eqref{SR_real},  we find  
\begin{equation}   \label{SR_dif_real}
\frac{t}{2}
\frac{\kappa_{n-1}^2}{\kappa_{n}^2}
\alpha_n =  -n  \alpha_{n-1} 
- \frac{t}{2} \frac{\kappa_{n-1}^2}{\kappa_{n}^2} \alpha_{n-2}.
\end{equation}   
Substituting \eqref{SR_dif_real},  \eqref{Szegorecorrence}, and 
$\Phi_{n-1}^*(z)=\Phi_{n-2}^*(z)-
\alpha_{n-2} z\Phi_{n-2}(z)$
in the equation \eqref{SR_real}, after some manipulation we obtain \eqref{rel_est_conhecida}. Hence, equation \eqref{SR_real} is equivalent to the structure relation \eqref{rel_est_conhecida}.

\section*{Acknowledgments} 

The authors are grateful to the anonymous referees for their comments and suggestions that helped to improve the manuscript.

This work is part of the doctoral thesis of the author K.S. Rampazzi at UNESP, S\~{a}o Jos\'{e} do Rio Preto, SP, Brazil, supported by a grant from CAPES, Brazil.

The research of the author C.F. Bracciali was supported by the grant 2022/09575-5 from FAPESP, 
S\~{a}o Paulo, Brazil.



\begin{thebibliography}{10}


\bibitem{AlMa91} 
M. Alfaro, F. Marcell\'an, Recent trends in orthogonal polynomials on the unit circle, In
Orthogonal Polynomials and their applications, Proceedings Erice 1990, C. Brezinski, L. Gori, and
A. Ronveaux Editors, pp. 3--14, IMACS Ann. Comput. Appl. Math., 9, Baltzer, Basel, 1991.

  
\bibitem{BLN87}  
S. Bonan, D. Lubinsky, P. Nevai, 
Orthogonal polynomials and their derivatives II, 
SIAM J. Math. Anal., 18 (1987) 1163--1176.

\bibitem{Bra96}      
A. Branquinho, 
A note on semi-classical orthogonal polynomials, 
Bull. Belg. Math. Soc., 3 (1996) 1--12.        
        
\bibitem{BraRe12}
A. Branquinho, M. N. Rebocho,
Structure relations for orthogonal polynomials on the unit circle,     
Linear Algebra Appl., 436 (2012) 4296--4310.      
        
\bibitem{CaSu97}
A. Cachafeiro, C. Su\'{a}rez,
About semiclassical polynomials on the unit circle corresponding to the class (2,2),
J. Comput. Appl. Math., 85 (1997) 123--144.

\bibitem{Fr76}
G. Freud,
On the coefficients in the recursion formulae of orthogonal polynomials. 
Proc. R. Irish Acad. Sect., A 76 (1976), 1--6.             
        
\bibitem{Is05}
M. E. H. Ismail, 
Classical and Quantum Orthogonal Polynomials in One Variable, 
Encyclopedia of Mathematics and its Applications, Vol. 98, Cambridge Univ. Press, Cambridge, 2005.

\bibitem{IsWi01}
M. E. H. Ismail, N. S. Witte, 
Discriminants and functional equations for polynomials
orthogonal on the unit circle, 
J. Approx. Theory, 110 (2001) 200--228.

\bibitem{Ma86} 
A. P. Magnus, 
On Freud's equations for exponential weights.  
J. Approx. Theory, 46 (1986) 64-99.

\bibitem{Ma00} 
A. P. Magnus, 
MAPA 3072A Special topics in approximation theory: Semi-classical orthogonal polynomials on the unit circle, March 2013. Available in
http://perso.uclouvain.be/alphonse.magnus.

\bibitem{MaSR17}
F. Marcell\'an, A. Sri Ranga,
Sobolev orthogonal polynomials on the unit circle and coherent pairs of measures of the second kind,
Results Math., 71 (2017) 1127--1149. 

\bibitem{Maroni84} 
P. Maroni, 
Sur quelques espaces de distributions qui sont des formes lin\'{e}aires sur l’espace vectoriel des polyn\^{o}mes, (French), In Orthogonal Polynomials and Applications (Bar-le-Duc, 1984), pp. 184--194, Lecture Notes in Math., 1171, Springer, Berlin, 1985.

\bibitem{Maroni90} 
P. Maroni, 
Une th\'{e}orie alg\'{e}brique des polyn\^{o}mes 
orthogonaux. Application aux polyn\^{o}mes orthogonaux semi-classiques, In Orthogonal Polynomials and their applications, Proceedings Erice 1990,  C. Brezinski, L. Gori, and A. Ronveaux Editors,  pp. 95--130, IMACS Ann. Comput. Appl. Math., 9, Baltzer, Basel, 1991.

\bibitem{PS90}
V. Periwal, D. Shevitz, 
Unitary-matrix models as exactly solvable string theories, 
Phys. Rev. Letters, 64 (1990) 1326--1329.

\bibitem{Simon-Book-p1} 
B. Simon, {Orthogonal Polynomials on the Unit Circle. Part 1. Classical Theory},  Amer. Math. Soc. Colloq. Publ., vol. 54, part 1, Amer. Math. Soc., Providence, RI,  2005.

\bibitem{Ra10}
A. Sri Ranga, 
Szeg\H{o} polynomials from hypergeometric functions, 
Proc. Amer. Math. Soc., 138 (2010) 4259--4270.

\bibitem{Su01}
C. Su\'{a}rez, 
About semiclassical orthogonal polynomials of the class (2,2), J. Comput. Appl. Math., 131 (2001) 457--472.

\bibitem{Su08}
C. Su\'{a}rez,
On a family of semiclassical orthogonal polynomials on the unit circle belonging to the class (p,p),
Integral Transforms Spec. Funct., 19 (2008) 333--349.


\bibitem{Va08} 
 W. Van Assche, 
Orthogonal Polynomials and Painlev\'e Equations,
Australian Math. Soc. Lecture Series, 
27, Cambridge Univ. Press, Cambridge, 2018.   


\end{thebibliography}
\end{document}